%% file: Mean_curvature_flow_of_Lagrangian_submanifolds_with_ICS.tex
\documentclass[10pt,a4paper]{article}

\usepackage{amsthm}
\usepackage{amssymb}
\usepackage{latexsym}
\usepackage{amsmath}
\usepackage{euscript}
\usepackage{graphicx}
\usepackage{fancyhdr}
\usepackage{calc}
\usepackage{color}
\usepackage{romannum}
\usepackage{paralist}

\newtheorem{Def}{Definition}[section]
\newtheorem{Lem}[Def]{Lemma}

\newtheorem{Thm}[Def]{Theorem}
\newtheorem{Prop}[Def]{Proposition}

\def\im{{\mathop{\rm im}}}
\def\Span{{\mathop{\rm span}}}

\def\trace{{\mathop{\rm tr}}}
\def\sindex{{\mathop{\rm s\mbox{-}index}}}

\def\realpart{{\mathop{\rm Re}}}
\def\imagpart{{\mathop{\rm Im}}}
\def\Hom{{\mathop{\rm Hom}}}
\def\loc{{\mathop{\rm loc}}}
\def\Ric{{\mathop{\rm Ric}}}
\def\cs{{\mathop{\rm cs}}}

\def\eq#1{{\rm(\ref{#1})}}

\begin{document}
\pagenumbering{arabic}

\title{Mean curvature flow of Lagrangian submanifolds with isolated conical singularities}
\author{Tapio Behrndt}
\maketitle

\begin{abstract}
In this paper we study the short time existence problem for the (generalized) Lagrangian mean curvature flow in (almost) Calabi--Yau manifolds when the initial Lagrangian submanifold has isolated conical singularities modelled on stable special Lagrangian cones. Given a Lagrangian submanifold $F_0:L\rightarrow M$ in an almost Calabi--Yau manifold $M$ with isolated conical singularities at $x_1,\ldots,x_n\in M$ modelled on stable special Lagrangian cones $C_1,\ldots,C_n$ in $\mathbb{C}^m$, we show that for a short time there exist one-parameter families of points $x_1(t),\ldots x_n(t)\in M$ and a one parameter family of Lagrangian submanifolds $F(t,\cdot):L\rightarrow M$ with isolated conical singularities at $x_1(t),\ldots,x_n(t)\in M$ modelled on $C_1,\ldots,C_n$, which evolves by (generalized) Lagrangian mean curvature flow with initial condition $F_0:L\rightarrow M$.
\end{abstract}

\section{Introduction}
\subsection{Lagrangian mean curvature flow}
In a Calabi--Yau manifold $M$ with holomorphic volume form $\Omega$ there is a distinguished class of submanifolds called special Lagrangian submanifolds. These are oriented Lagrangian submanifolds that are calibrated with respect to $\realpart\;\Omega$. There has been growing interest in special Lagrangian submanifolds in the past decade since these are the key ingredient in the Strominger--Yau--Zaslow conjecture \cite{SYZ} which states mirror symmetry in terms of special Lagrangian torus fibrations.

Proving the existence of special Lagrangian submanifolds in a Calabi--Yau manifold is a hard problem. For instance Wolfson proved in \cite{Wolfson} the existence of a K3-surface which has no special Lagrangian submanifolds. This shows how subtle the issue is. However, since special Lagrangian submanifolds are calibrated submanifolds, they are volume minimizers in their homology class. One possible approach to the study of the existence of special Lagrangian submanifolds is therefore through mean curvature flow, which is the negative gradient flow of the volume functional. The key observation here is due to Smoczyk \cite{Smoczyk} who proves that a compact Lagrangian submanifold in a Calabi--Yau manifold (or even in a K\"ahler--Einstein manifold) remains Lagrangian under the mean curvature flow. The na\"{\i}ve idea is therefore to start with a Lagrangian submanifold in a Calabi--Yau manifold and to deform it under Lagrangian mean curvature flow to a special Lagrangian submanifold. The longtime convergence of the Lagrangian mean curvature flow to a special Lagrangian submanifold has so far only been verified in several special cases, see for instance Smoczyk and Wang \cite{SmoczykWang} and Wang \cite{Wang}. Also in \cite{ThomasYau} Thomas and Yau conjecture that for a given Lagrangian submanifold in a Calabi--Yau manifold, which satisfies a certain stability condition, the Lagrangian mean curvature flow exists for all time and converges to a special Lagrangian submanifold. In general however one expects that a Lagrangian submanifold will form a finite time singularity under the mean curvature flow. In fact, recently Neves \cite{Neves} constructed examples of Lagrangian surfaces in two dimensional Calabi--Yau manifolds which develop a finite time singularity under the mean curvature flow. The appearance of finite time singularities in the Lagrangian mean curvature flow therefore seems to be unavoidable in general.

When a finite time singularity occurs there are two possibilities, depending on the kind of singularity, how the flow can be continued. The first possibility is as in Perelman's work \cite{Perelman} on the Ricci flow of three manifolds, where a surgery is performed before the singularity occurs and the flow is then continued. The other possibility to continue the Lagrangian mean curvature flow when a finite time singularity occurs is to evolve the singular Lagrangian submanifold by mean curvature flow in a specific class of singular Lagrangian submanifolds. 

\subsection{Results and overview of this paper}
In this paper we study the short time existence problem for the generalized Lagrangian mean curvature flow in almost Calabi--Yau manifolds when the initial Lagrangian submanifold has isolated conical singularities modelled on stable special Lagrangian cones. We show that for a given Lagrangian submanifold $F_0:L\rightarrow M$ with isolated conical singularities modelled on stable special Lagrangian cones one can find for a short time a solution $F(t,\cdot):L\rightarrow M$, $0\leq t<T$, to the generalized Lagrangian mean curvature flow with initial condition $F_0:L\rightarrow M$, by letting the conical singularities move around in $M$. The Lagrangian mean curvature flow of $F_0:L\rightarrow M$ (here on the left) looks therefore after a short time like the surface on the right.

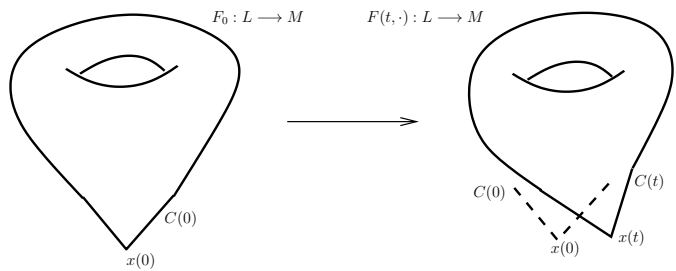
\begin{center}
\scalebox{0.5}{{\input{LMCF.pstex_t}}}
\end{center}

We give a short overview of this paper. In Section \ref{SUPERMANN} we first introduce some necessary background material from symplectic geometry and Riemannian submanifold geometry. Further we define the notion of the generalized Lagrangian mean curvature flow in almost Calabi--Yau manifolds and present a new approach to the short time existence problem of the generalized Lagrangian mean curvature flow when the initial Lagrangian submanifold is a compact Lagrangian submanifold. We feel that it is helpful first to understand our alternative approach to the short time existence problem when the initial submanifold is compact in order to understand the much more complicated approach to the short time existence problem when the initial Lagrangian submanifold has isolated conical singularities. In Section \ref{ANALYSIS} we review some important results about linear parabolic equations on Riemannian manifolds with conical singularities, which build the core for the later short time existence proof. In Section \ref{SUPERMANNOOO} we then introduce special Lagrangian cones and Lagrangian submanifolds with isolated conical singularities in almost Calabi--Yau manifolds. Further we discuss several Lagrangian neighbourhood theorems that will assist us in setting up the later short time existence problem. Finally in Section \ref{DITISKLASSE} we discuss the short time existence proof of the Lagrangian mean curvature flow when the initial Lagrangian submanifold has isolated conical singularities modelled on stable special Lagrangian cones. First, generalizing the ideas from Section \ref{SUPERMANN}, we discuss how to set up the short time existence problem using a Lagrangian neighbourhood theorem for Lagrangian submanifolds with isolated conical singularities. Then, using the analytical results from Section \ref{ANALYSIS}, we discuss in an informal way how the short time existence of the flow is proven. We avoid the long and rather complicated analytical details of the short time existence proof and merely concentrate on the ideas of the proof. The interested reader may consult the author's DPhil thesis \cite{Behrndt1} to learn about the details of the proof.

\subsection{Acknowledgements}

The author wishes to thank his supervisor Dominic Joyce for many useful discussions about his work on special Lagrangian submanifolds with isolated conical singularities. The author is also grateful to Tom Ilmanen and Andr\'e Neves for interesting conversations. This work was supported by a Sloane Robinson Foundation Graduate Award of the Lincoln College and by an EPSRC Research Studentship.

\section{Generalized Lagrangian mean curvature flow of compact Lagrangian submanifolds}\label{SUPERMANN}

Before we begin with our review of some basic notions from Riemannian submanifold geometry and symplectic geometry we have to make a remark about the regularity of the manifolds and maps, i.e. functions, differential forms, vector fields, and embeddings, that we consider in this paper. All the manifolds that we consider in this paper are assumed to be smooth and connected. Moreover we make use of the convention that all the maps we are considering are smooth, unless differently specified. For example, when we say that $u$ is a function on the manifold $M$ or $\beta$ is a one-form on $M$, then we mean that $u$ is a smooth function on $M$ and $\beta$ is a smooth one-form on $M$. Otherwise we may say that $u$ is a $C^k$-function on $M$, meaning that $u$ is $k$-times continuously differentiable.

Throughout this section we will restrict ourselves to smooth maps. The definitions and results that we present in this section, however, have straightforward generalizations when the maps have less regularity (assuming that the maps are $C^2$ is usually sufficient). We wanted to mention this rather obvious fact, since the regularity of the maps is of particular importance in the study of the short time existence problem in \S\ref{DITISKLASSE}.

\subsection{Some notions from Riemannian submanifold geometry}
We now recall some basic definitions from Riemannian submanifold geometry. 

Let $(M,g)$ be an $m$-dimensional Riemannian manifold and $N$ a manifold of dimension $n$ with $n\leq m$. An embedding of $N$ into $M$ is an injective map $F:N\rightarrow M$, such that the differential $\mathrm dF(x):T_xN\rightarrow T_{F(x)}M$ is injective for every $x\in N$. The image $F(N)$ of an embedding $F:N\rightarrow M$ is then an $n$-dimensional submanifold of $M$. In this paper we will refer to an embedding $F:N\rightarrow M$ as an $n$-dimensional submanifold of $M$.

A submanifold $F:N\rightarrow M$ defines an orthogonal decomposition of the vector bundle $F^*(TM)$ into $\mathrm dF(TN)\oplus \nu N$. The vector bundle $\nu N$ over $N$ is the normal bundle of $F:N\rightarrow M$. Denote by $\pi_{\nu N}$ the orthogonal projection $F^*(TM)\rightarrow \nu N$ onto the normal bundle of $F:N\rightarrow M$. The second fundamental form of a submanifold $F:N\rightarrow M$ is a section of the vector bundle $\odot^2T^*N\otimes \nu N$ defined by $\mbox{\Romannum{2}}(X,Y)=\pi_{\nu N}(\nabla_{\mathrm dF(X)}\mathrm dF(Y))$ for $X,Y\in TN$. Here $\nabla$ is the Levi--Civita connection of $g$. The mean curvature vector field of $F:N\rightarrow M$ is a section of $\nu N$ defined by $H=\trace\;\mbox{\Romannum{2}}$, where the trace is taken with respect to the Riemannian metric $F^*(g)$ on $N$. Finally, a submanifold $F:N\rightarrow M$ is a minimal submanifold if the mean curvature vector field is zero. It can be shown that a compact submanifold $F:N\rightarrow M$ is minimal if and only if it is a critical point of the volume functional.

\subsection{Symplectic manifolds and Lagrangian submanifolds}
\begin{Def}
A $2m$-dimensional symplectic manifold is a pair $(M,\omega)$, where $M$ is a $2m$-dimensional manifold and $\omega$ is a closed and non-degenerate two-form on $M$. 
\end{Def}

The most elementary example of a symplectic manifold is $(\mathbb{C}^m,\omega')$, where $\omega'=\sum_{j=1}^m\mathrm dx_j\wedge \mathrm dy_j$, and $(x_1,\ldots,y_m)$ are the usual real coordinates on $\mathbb{C}^m$. Denote by $B_R$ the open ball of radius $R>0$ about the origin in $\mathbb{C}^m$. Then $(B_R,\omega')$ is a symplectic manifold, and in fact every $2m$-dimensional symplectic manifold is locally isomorphic to $(B_R,\omega')$ for some small $R>0$. This is the statement of Darboux' Theorem \cite[Thm 3.15]{McDuffSalamon2}.

\begin{Thm}\label{DarbouxTheorem}
Let $(M,\omega)$ be a $2m$-dimensional symplectic manifold, $x\in M$, and let $A:\mathbb{C}^m\rightarrow T_xM$ be an isomorphism with $A^*(\omega)=\omega'$. Then there exists $R>0$ and an embedding $\Upsilon:B_R\rightarrow M$, such that $\Upsilon^*(\omega)=\omega'$, $\Upsilon(0)=x$, and $\mathrm d\Upsilon(0)=A$.
\end{Thm}

Another important example of a symplectic manifold is the cotangent bundle of a manifold. If $M$ is an $m$-dimensional manifold, then the cotangent bundle $T^*M$ of $M$ is a $2m$-dimensional manifold that has a canonical symplectic structure $\hat{\omega}$ defined as follows. Denote by $\pi:T^*M\rightarrow M$ the canonical projection and let $\hat{\lambda}$ be the one-form on $T^*M$ defined by $\hat{\lambda}(\beta)=(\mathrm d\pi)^*(\beta)$ for $\beta\in T^*M$. Set $\hat{\omega}=-\mathrm d\hat{\lambda}$, then one can show that $\hat{\omega}$ is a symplectic structure on $T^*M$.

\begin{Def}
Let $(M,\omega)$ be an $2m$-dimensional symplectic manifold. An $m$-dimensional submanifold $F:L\rightarrow M$ of $M$ is a Lagrangian submanifold if $F^*(\omega)=0$.
\end{Def}

Of particular importance for our later study of the generalized Lagrangian mean curvature flow is the notion of a Lagrangian neighbourhood, which we now introduce.

\begin{Def}
Let $(M,\omega)$ be a symplectic manifold and $F:L\rightarrow M$ a Lagrangian submanifold of $M$. A Lagrangian neighbourhood for $F:L\rightarrow M$ is an embedding $\Phi_L:U_L\rightarrow M$ of an open neighbourhood $U_L$ of the zero section in $T^*L$ onto an open neighbourhood of $F(L)$ in $M$, such that $\Phi_L^*(\omega)=\hat{\omega}$ and $\Phi_L(x,0)=F(x)$ for $x\in L$.
\end{Def}

When $F:L\rightarrow M$ is a compact Lagrangian submanifold, then the existence of a Lagrangian neighbourhood for $F:L\rightarrow M$ is guaranteed by the Lagrangian Neighbourhood Theorem.

\begin{Thm}[\bf{Lagrangian Neighbourhood Theorem}]\label{LagrangianNeighbourhood}
Let $(M,\omega)$ be a symplectic manifold and $F:L\rightarrow M$ a compact Lagrangian submanifold. Then there exists a Lagrangian neighbourhood $\Phi_L:U_L\rightarrow M$ for $F:L\rightarrow M$. 
\end{Thm}

\noindent A proof of the Lagrangian Neighbourhood Theorem for compact Lagrangian submanifolds can be found in McDuff and Salamon \cite[Thm. 3.32]{McDuffSalamon2}.

\subsection{Almost Calabi--Yau manifolds}

We define almost Calabi--Yau manifolds following Joyce \cite[Def. 8.4.3]{JoyceBook}.

\begin{Def}
An $m$-dimensional almost Calabi--Yau manifold is a quadruple $(M,J,\omega,\Omega)$, where $(M,J)$ is an $m$-dimensional complex manifold, $\omega$ is the K\"ahler form of a K\"ahler metric $g$ on $M$, and $\Omega$ is a holomorphic volume form on $M$.
\end{Def}

Let $(M,J,\omega,\Omega)$ be an $m$-dimensional almost Calabi--Yau manifold. The Ricci-form is the complex $(1,1)$-form given by $\rho(X,Y)=\Ric(JX,Y)$ for $X,Y\in TM$, where $\Ric$ is the Ricci-tensor of $g$. We define a function $\psi$ on $M$ by
\begin{equation}\label{DefinitionPsi}
e^{2m\psi}\frac{\omega^m}{m!}=(-1)^{\frac{m(m-1)}{2}}\left(\frac{i}{2}\right)^m\Omega\wedge\bar{\Omega}.
\end{equation}
Then $|\Omega|=2^{m/2}e^{m\psi}$, so that $\Omega$ is parallel if and only if $\psi$ is constant. One can show that the Ricci-form of an almost Calabi--Yau manifold satisfies $\rho=\mathrm d\mathrm d^c\log|\Omega|$. Thus $\rho=m\mathrm d\mathrm d^c\psi$ and it follows that $g$ is Ricci-flat if and only if $\psi$ is constant. If $\psi\equiv 0$, then $(M,J,\omega,\Omega)$ is a Calabi--Yau manifold \cite[Ch. 8, \S 4]{JoyceBook}. 

The most important example of an (almost) Calabi--Yau manifold is $\mathbb{C}^m$ with its standard structure. Denote by $(x_1,\ldots,x_m,y_1,\ldots,y_m)$ the usual real coordinates on $\mathbb{C}^m$. We define a complex structure $J'$, a non-degenerate two form $\omega'$, and a holomorphic volume form $\Omega'$ on $\mathbb{C}^m$ by
\begin{eqnarray*}
&&J'\left(\frac{\partial}{\partial x_j}\right)=\frac{\partial}{\partial y_j}\quad\mbox{and } J'\left(\frac{\partial}{\partial y_j}\right)=-\frac{\partial}{\partial x_j}\quad\mbox{for }j=1,\ldots,m,\\
&&\omega'=\sum_{j=1}^m\mathrm dx_j\wedge \mathrm dy_j,\quad \Omega'=(\mathrm dx_1+i\mathrm dy_1)\wedge\cdots\wedge(\mathrm dx_m+i\mathrm dy_m).
\end{eqnarray*}
Then $(\mathbb{C}^m,J',\omega',\Omega')$ is an (almost) Calabi--Yau manifold and the corresponding Riemannian metric is the Euclidean metric $g'=\mathrm dx_1^2+\cdots+\mathrm dy_m^2$.

We now discuss Lagrangian submanifolds in almost Calabi--Yau manifolds. Thus let $(M,J,\omega,\Omega)$ be an $m$-dimensional almost Calabi--Yau manifold and $F:L\rightarrow M$ a Lagrangian submanifold. We define a section $\alpha$ of the vector bundle $\Hom(\nu L,T^*L)$ by
\begin{equation}\label{Defalpha}
\alpha(\xi)=\alpha_{\xi}=F^*(\xi\;\lrcorner\;\omega)\quad\mbox{for }\xi\in\nu L.
\end{equation}
Since $F:L\rightarrow M$ is Lagrangian, $\alpha$ is an isomorphism in each fibre over $L$. Moreover, $\alpha^{-1}(\mathrm du)=-J(\mathrm dF(\nabla u))$ for every function $u$ on $L$.

Let $H$ be the mean curvature vector field of $F:L\rightarrow M$. The one-form $\alpha_H=F^*(H\;\lrcorner\;\omega)$ on $L$ is the mean curvature form of $F:L\rightarrow M$. Then $\mathrm d\alpha_H=F^*(\rho)$, as first observed by Dazord \cite{Dazord}. Assume for the moment that $(M,J,\omega,\Omega)$ is Calabi--Yau. Then $\rho\equiv 0$, as $g$ is Ricci-flat. In particular $\alpha_H$ is closed and it follows from Cartan's formula that $F^*(\mathcal{L}_H\omega)=0$. Thus, if $(M,J,\omega,\Omega)$ is Calabi--Yau, then the deformation of a Lagrangian submanifold in direction of the mean curvature vector field is an infinitesimal symplectic motion. Now if $(M,J,\omega,\Omega)$ is an almost Calabi--Yau manifold, then the Ricci-form is given by $\rho=m\mathrm d\mathrm d^c\psi$. In particular $F^*(\mathcal{L}_H\omega)=mF^*(\mathrm d\mathrm d^c\psi)$ is nonzero in general. We therefore need a generalization of the mean curvature vector field with the property that the deformation of a Lagrangian submanifold in its direction is an infinitesimal symplectic motion. This leads to the definition of the generalized mean curvature vector field, which was introduced by the author in \cite[\S 3]{Behrndt} and later generalized by Smoczyk and Wang in \cite{SmoczykWang2}.

\begin{Def}\label{GenMCF}
The generalized mean curvature vector field of $F:L\rightarrow M$ is the normal vector field $K=H-m\pi_{\nu L}(\nabla\psi)$, where $H$ denotes the mean curvature vector field of $F:L\rightarrow M$. The one-form $\alpha_K=F^*(K\;\lrcorner\;\omega)$ is the generalized mean curvature form of $F:L\rightarrow M$.
\end{Def}

Note that if $\psi$ is constant, then $K\equiv H$. Furthermore, if $F:L\rightarrow M$ is Lagrangian, then a short calculation shows that $F^*(\mathcal{L}_K\omega)=0$. Thus if $F:L\rightarrow M$ is a Lagrangian submanifold in an almost Calabi--Yau manifold, then the deformation of $F:L\rightarrow M$ in the direction of the generalized mean curvature vector field is an infinitesimal symplectic motion.

Next we define the Lagrangian angle of a Lagrangian submanifold. Thus let $F:L\rightarrow M$ be a Lagrangian submanifold. The Lagrangian angle of $F:L\rightarrow M$ is the map $\theta(F):L\rightarrow\mathbb{R}/\pi\mathbb{Z}$ defined by
\begin{equation*}
F^*(\Omega)=e^{i\theta(F)+mF^*(\psi)}\mathrm dV_{F^*(g)}.
\end{equation*}
Since $F:L\rightarrow M$ is a Lagrangian submanifold, $\theta(F)$ is in fact well defined, see for instance Harvey and Lawson \cite[III.1]{HarveyLawson}. In general $\theta(F):L\rightarrow\mathbb{R}/\pi\mathbb{Z}$ cannot be lifted to a smooth function $\theta(F):L\rightarrow\mathbb{R}$. However, $\mathrm d[\theta(F)]$ is a well defined closed one-form on $L$, so it represents a cohomology class $\mu_F\in H^1(L,\mathbb{R})$ in the first de Rham cohomology group of $L$. Thus if $\mu_F=0$, then $\theta(F):L\rightarrow\mathbb{R}/\pi\mathbb{Z}$ can be lifted to a smooth function $\theta(F):L\rightarrow\mathbb{R}$ and vice versa. The cohomology class $\mu_F$ is called the Maslov class of $F:L\rightarrow M$.

The following proposition gives an important relation between the generalized mean curvature form of a Lagrangian submanifold $F:L\rightarrow M$ and the Lagrangian angle.

\begin{Prop}\label{GeneralizedMeanCurvatureForm}
Let $F:L\rightarrow M$ be a Lagrangian submanifold in an almost Calabi--Yau manifold. Then the generalized mean curvature form of $F:L\rightarrow M$ satisfies $\alpha_K=-\mathrm d[\theta(F)]$.
\end{Prop}

\noindent A proof of Proposition \ref{GeneralizedMeanCurvatureForm} can be found in the author's paper \cite[Prop. 4]{Behrndt}.

Notice that as a consequence of Proposition \ref{GeneralizedMeanCurvatureForm}, if $F:L\rightarrow M$ is a Lagrangian submanifold with zero Maslov class, then $\alpha_K$ is an exact one-form and the deformation of $F:L\rightarrow M$ in direction of the generalized mean curvature vector field is an infinitesimal Hamiltonian motion.

Next we define a special class of Lagrangian submanifolds in almost Calabi--Yau manifolds called special Lagrangian submanifolds. 

\begin{Def}\label{DefinitionSpecialLagrangian}
Let $F:L\rightarrow M$ be a Lagrangian submanifold in an almost Calabi--Yau manifold $(M,J,\omega,\Omega)$. Then $F:L\rightarrow M$ is a special Lagrangian submanifold with phase $e^{i\theta}$, $\theta\in\mathbb{R}$, if and only if
\[
F^*(\cos\theta\;\imagpart\;\Omega-\sin\theta\;\realpart\;\Omega)=0.
\]
If $F:L\rightarrow M$ is a special Lagrangian submanifold with phase $e^{i\theta}$, then there is a unique orientation on $L$ in which $F^*(\cos\theta\;\realpart\;\Omega+\sin\theta\;\imagpart\;\Omega)$ is positive.
\end{Def}

\noindent Note that a special Lagrangian submanifold $F:L\rightarrow M$ has zero Maslov-class, since $\theta(F)$ is constant on $L$ and $\mathrm d[\theta(F)]$ represents $\mu_F$ by Proposition \ref{GeneralizedMeanCurvatureForm}. In particular it follows from Proposition \ref{GeneralizedMeanCurvatureForm} that special Lagrangian submanifolds are minimal.

Definition \ref{DefinitionSpecialLagrangian} is not the usual definition of special Lagrangian submanifolds in terms of calibrations, as defined by Harvey and Lawson in \cite{HarveyLawson}. Our definition is, however, equivalent to the definition of special Lagrangian submanifolds as a special class of calibrated submanifolds. Let us show how Definition \ref{DefinitionSpecialLagrangian} fits into the usual frame of special Lagrangian submanifolds as calibrated submanifolds. If we define $\tilde{g}$ to be the conformally rescaled Riemannian metric on $M$ given by $\tilde{g}=e^{2\psi}g$, then one can show that $\realpart\;\Omega$ is a calibration on the Riemannian manifold $(M,\tilde{g})$. We then have the following alternative characterization of special Lagrangian submanifolds.

\begin{Prop}
Let $F:L\rightarrow M$ be an oriented Lagrangian submanifold of an almost Calabi--Yau manifold $(M,J,\omega,\Omega)$. Then $F:L\rightarrow M$ is a special Lagrangian submanifold with phase $e^{i\theta}$, $\theta\in\mathbb{R}$, if and only if $F:L\rightarrow M$ is calibrated with respect to $\realpart(e^{-i\theta}\Omega)$ for the metric $\tilde{g}$.
\end{Prop}

\subsection{Lagrangian submanifolds in the cotangent bundle}

Let $(M,\omega)$ be a $2m$-dimensional symplectic manifold and $L$ an $m$-dimensional manifold. Let $T^*L$ be the cotangent bundle of $L$ and $\beta$ a one-form on $L$. The graph of $\beta$ is the submanifold
\begin{equation*}
\hat{F}:L\longrightarrow T^*L,\quad\hat{F}(x)=(x,\beta(x))\in T^*_xL\quad\mbox{for }x\in L.
\end{equation*}
We write $\Gamma_{\beta}$ for $\hat{F}(L)=\{(x,\beta(x)):x\in L\}$. Then $\hat{F}^*(\hat{\omega})=-\mathrm d\beta$, so that $\hat{F}:L\rightarrow T^*L$ is a Lagrangian submanifold of $T^*L$ if and only if $\beta$ is closed. In particular every function $u$ on $L$ defines a Lagrangian submanifold $\hat{F}:L\rightarrow T^*L$ by $\hat{F}(x)=(x,\mathrm du(x))$ for $x\in L$.

Now let $F:L\rightarrow M$ be a Lagrangian submanifold and assume that we are given a Lagrangian neighbourhood $\Phi_L:U_L\rightarrow M$ for $F:L\rightarrow M$. If $\beta$ is a closed one-form on $L$ with $\Gamma_{\beta}\subset U_L$, then we can define a submanifold by
\[
\Phi_L\circ\beta:L\longrightarrow M,\quad(\Phi_L\circ\beta)(x)=\Phi_L(x,\beta(x))\quad\mbox{for }x\in L.
\]
Since $\Phi_L^*(\omega)=\hat{\omega}$ and $\beta$ is closed, $\Phi_L\circ\beta:L\rightarrow M$ is a Lagrangian submanifold. Note that if $L$ is compact, then, after reparametrizing by a diffeomorphism on $L$, every Lagrangian submanifold $\tilde{F}:L\rightarrow M$ that is $C^1$-close to $F:L\rightarrow M$ is given by $\Phi_L\circ\beta:L\rightarrow M$ for some unique closed one-form $\beta$ on $L$. 

When we study the generalized Lagrangian mean curvature flow as a flow of functions, we will study deformations of Lagrangian submanifolds of the form $\Phi_L\circ(\beta+s\eta):L\rightarrow M$, for small $s\in\mathbb{R}$ and $\beta,\eta$ closed one-forms on $L$ with $\Gamma_{\beta}\subset U_L$. The next lemma gives a formula for the variation vector field of $\Phi_L\circ(\beta+s\eta):L\rightarrow M$ along the submanifold $\Phi_L\circ\beta:L\rightarrow M$.

\begin{Lem}\label{Var}
Let $\beta,\eta$ be closed one-forms on $L$ with $\Gamma_{\beta}\subset U_L$ and $\varepsilon>0$ sufficiently small such that $\Gamma_{\beta+s\eta}\subset U_L$ for $s\in(-\varepsilon,\varepsilon)$. Then for every $s\in(-\varepsilon,\varepsilon)$, $\Phi_L\circ(\beta+s\eta):L\rightarrow M$ is a Lagrangian submanifold and
\begin{equation*}
\frac{\mathrm d}{\mathrm ds}\Phi_L\circ(\beta+s\eta)\Big|_{s=0}=-\alpha^{-1}(\eta)+V(\eta),
\end{equation*}
where $\alpha$ is defined in \eq{Defalpha} and $V(\eta)=\mathrm d(\Phi_L\circ\beta)(\hat{V}(\eta))$, $\hat{V}(\eta)\in TL$, is the tangential part of the variation vector field.
\end{Lem}

\subsection{Generalized Lagrangian mean curvature flow}
\begin{Def}
Let $F_0:L\rightarrow M$ be a Lagrangian submanifold of $M$. A one-parameter family $\{F(t,\cdot)\}_{t\in(0,T)}$ of Lagrangian submanifolds $F(t,\cdot):L\rightarrow M$, which is continuous up to $t=0$, is evolving by generalized Lagrangian mean curvature flow with initial condition $F_0:L\rightarrow M$ if
\begin{equation}\label{GMCF}
\begin{split}
&\pi_{\nu L}\left(\frac{\partial F}{\partial t}\right)(t,x)=K(t,x)\quad\mbox{for }(t,x)\in(0,T)\times L,\\&F(0,x)=F_0(x)\quad\quad\quad\quad\quad\;\;\mbox{for }x\in L.
\end{split}
\end{equation}
Here $K(t,\cdot)$ is the generalized mean curvature vector field of $F(t,\cdot):L\rightarrow M$ for $t\in(0,T)$ as in Definition \ref{GenMCF}. If $M$ is Calabi--Yau, then $\psi\equiv 0$ and $K\equiv H$. Then we say that $\{F(t,\cdot)\}_{t\in(0,T)}$ evolves by Lagrangian mean curvature flow.
\end{Def}

\noindent We will establish the short time existence of solutions to the generalized Lagrangian mean curvature flow \eq{GMCF} when the initial Lagrangian submanifold $F_0:L\rightarrow M$ is compact in Theorem \ref{GMCFShortTime} below. We first give a short general discussion of the generalized Lagrangian mean curvature flow.

The system of partial differential equations in \eq{GMCF} is, after reparametrizing by a family of diffeomorphisms on $L$, a quasilinear parabolic system. Hence, if $L$ is compact, then it follows from the standard theory for parabolic equations on compact manifolds, see for instance Aubin \cite[\S 4.2]{Aubin}, that for every submanifold $F_0:L\rightarrow M$ there exists a one-parameter family $\{F(t,\cdot)\}_{t\in(0,T)}$ of submanifolds $F(t,\cdot):L\rightarrow M$, which is continuous up to $t=0$ and satisfies \eq{GMCF}. Less obvious, however, is the fact that if $F_0:L\rightarrow M$ is a Lagrangian submanifold, then $F(t,\cdot):L\rightarrow M$ is a Lagrangian submanifold for every $t\in(0,T)$. The original proof of the fact that $F(t,\cdot):L\rightarrow M$ is a Lagrangian submanifold for $t\in(0,T)$ uses long computations in local coordinates and the parabolic maximum principle. In \S\ref{GULLI} we show how the generalized Lagrangian mean curvature flow can be integrated to a flow of functions on $L$ rather than of embeddings of $L$ into $M$. Using this interpretation of the generalized Lagrangian mean curvature flow we present in \S\ref{GULLI} a new short time existence proof for the generalized Lagrangian mean curvature flow when $F_0:L\rightarrow M$ is compact

The idea of the Lagrangian mean curvature flow goes already back to Oh \cite{Oh} in the early nineties. The existence of the Lagrangian mean curvature flow, however, was first proved by Smoczyk \cite[Thm. 1.9]{Smoczyk} for the case when $M$ is a K\"ahler--Einstein manifold. Recently there has been interest in generalizing the idea of the Lagrangian mean curvature flow. This led to the notion of generalized Lagrangian mean curvature flows first introduced by the author in \cite{Behrndt}, when $M$ is a K\"ahler manifold that is almost Einstein, and later by Smoczyk and Wang \cite{SmoczykWang2}, when $M$ is an almost K\"ahler manifold that admits an Einstein connection.

The next proposition discusses another definition of the generalized Lagrangian mean curvature flow, which at least in the case when $F:L\rightarrow M$ is a compact Lagrangian submanifold, is equivalent to the previous one.

\begin{Prop}\label{GMLCF2}
Let $F_0:L\rightarrow M$ be a compact Lagrangian submanifold, and $\{F(t,\cdot)\}_{(0,T)}$ a one-parameter family of Lagrangian submanifolds $F(t,\cdot):L\rightarrow M$, which is continuous up to $t=0$ and evolves by generalized Lagrangian mean curvature flow with initial condition $F_0:L\rightarrow M$. Then there exists a one-parameter family $\{\varphi(t,\cdot)\}_{t\in(0,T)}$ of diffeomorphisms of $L$, which is continuous up to $t=0$, such that the following holds. The map $\varphi(0,\cdot):L\rightarrow L$ is the identity on $L$ and, if we define a one-parameter family $\{\tilde{F}(t,\cdot)\}_{t\in(0,T)}$ of Lagrangian submanifolds $\tilde{F}(t,\cdot):L\rightarrow M$ by
\begin{equation*}
\tilde{F}(t,x)=F(t,\varphi(t,x))\quad\mbox{for }(t,x)\in(0,T)\times L,
\end{equation*}
then $\{\tilde{F}(t,\cdot)\}_{t\in(0,T)}$ is continuous up to $t=0$ and satisfies 
\begin{equation}\label{GMCF2}
\begin{split}
&\frac{\partial\tilde{F}}{\partial t}(t,x)=K(t,x)\quad\mbox{for }(t,x)\in (0,T)\times L,\\&\tilde{F}(0,x)=F_0(x)\quad\quad\;\mbox{for }x\in L.
\end{split}
\end{equation}
\end{Prop}

Often \eq{GMCF2} is used for the definition of the generalized Lagrangian mean curvature flow. Proposition \ref{GMLCF2} shows that \eq{GMCF} and \eq{GMCF2} are equivalent up to a family of tangential diffeomorphisms, provided $L$ is compact. It is important to note, however, that in general \eq{GMCF} and \eq{GMCF2} are not equivalent. For instance in the generalized Lagrangian mean curvature flow with isolated conical singularities, which we study in \S\ref{DITISKLASSE}, we will find a solution to \eq{GMCF}. The solution will then consist of Lagrangian submanifolds with isolated conical singularities and the singularities move around in the ambient space. In this case it is in general not possible to reparametrize a solution of \eq{GMCF} by diffeomorphisms on $L$ in order to get a solution of \eq{GMCF2}. Note that if we are given solutions $\{F(t,\cdot)\}_{(0,T)}$ to the generalized Lagrangian mean curvature flow \eq{GMCF} and $\{\tilde{F}(t,\cdot)\}_{t\in(0,T)}$ to \eq{GMCF2}, then $F(t,L)=\tilde{F}(t,L)$ for $t\in(0,T)$. So $F(t,\cdot):L\rightarrow M$ and $\tilde{F}(t,\cdot):L\rightarrow M$ have the same image for each $t\in(0,T)$.

\subsection{Short time existence of the flow}\label{GULLI}

The short time existence of the generalized Lagrangian mean curvature flow when the initial Lagrangian submanifold is compact is established in the following theorem.

\begin{Thm}\label{GMCFShortTime}
Let $F_0:L\rightarrow M$ be a compact Lagrangian submanifold in an almost Calabi--Yau manifold $M$. Then there exists $T>0$ and a one-parameter family $\{F(t,\cdot)\}_{t\in(0,T)}$ of Lagrangian submanifolds $F(t,\cdot):L\rightarrow M$, which is continuous up to $t=0$ and evolves by generalized Lagrangian mean curvature flow with initial condition $F_0:L\rightarrow M$.
\end{Thm}

As mentioned before we now present a new proof of Theorem \ref{GMCFShortTime}. The idea of the proof is based on two observations. Firstly, when $F_0:L\rightarrow M$ is a compact Lagrangian submanifold, then by Theorem \ref{LagrangianNeighbourhood} there exists a Lagrangian neighbourhood $\Phi_L:U_L\rightarrow M$ of $F_0:L\rightarrow M$ and every Lagrangian submanifold $\tilde{F}:L\rightarrow M$ that is $C^1$-close to $F_0:L\rightarrow M$ is, after reparametrizing by a diffeomorphism on $L$, given by $\Phi_L\circ\beta:L\rightarrow M$ for some unique closed one-form $\beta$ on $L$. Secondly, by Proposition \ref{GeneralizedMeanCurvatureForm} the generalized mean curvature form of $F_0:L\rightarrow M$ satisfies $\alpha_K=-\mathrm d[\theta(F_0)]$. Assume for the moment that $F_0:L\rightarrow M$ has zero Maslov class, then $\alpha_K$ is exact and the Lagrangian mean curvature flow (if it exists) is a Hamiltonian deformation. Therefore we expect that the Lagrangian mean curvature flow of $F_0:L\rightarrow M$ (if it exists) should be equivalent to the existence of a solution to an evolution equation of the form $\partial_tu=\theta(\Phi_L\circ\mathrm du)$ for a function $u$ on $(0,T)\times L$ for $T>0$.

Let us now carry out these ideas in more detail. To this end let $F_0:L\rightarrow M$ be a compact Lagrangian submanifold in an almost Calabi--Yau manifold $M$ and let $\Phi_L:U_L\rightarrow M$ be a Lagrangian neighbourhood for $F_0:L\rightarrow M$ as given by Theorem \ref{LagrangianNeighbourhood}. Let $\mu_{F_0}$ be the Maslov class of $F_0:L\rightarrow M$, and choose a smooth map $\alpha_0:L\rightarrow\mathbb{R}/\pi\mathbb{Z}$ with $\mathrm d\alpha_0\in\mu_{F_0}$. Denote $\beta_0=\mathrm d\alpha_0$. Then we can choose a smooth lift $\Theta(F_0):L\rightarrow\mathbb{R}$ of $\theta(F_0)-\alpha_0:L\rightarrow\mathbb{R}/\pi\mathbb{Z}$. In particular $\Theta(F_0)$ satisfies $\mathrm d[\Theta(F_0)]=\mathrm d[\theta(F_0)]-\beta_0$. Moreover, if $\{\eta(s)\}_{s\in(-\varepsilon,\varepsilon)}$, $\varepsilon>0$, is a smooth family of closed one-forms defined on $L$ with $\Gamma_{\eta(s)}\subset U_L$ for $s\in(-\varepsilon,\varepsilon)$ and $\eta(0)=0$, then we can choose $\Theta(\Phi_L\circ\eta(s))$ to depend smoothly on $s\in(-\varepsilon,\varepsilon)$. 
 
We now define a nonlinear differential operator $P$ as follows. Define a one-parameter family $\{\beta(t)\}_{t\in(0,T)}$ of closed one-forms on $L$ by $\beta(t)=t\beta_0$ for $t\in(0,T)$. Then $\{\beta(t)\}_{t\in(0,T)}$ extends continuously to $t=0$ with $\beta(0)=0$. Choose $T>0$ small enough so that $\Gamma_{\beta(t)}\subset U_L$ for $t\in(0,T)$, and define the domain of $P$ by
\begin{align*}
\mathcal{D}&=\bigl\{u\in C^{\infty}((0,T)\times L)\;:\;u\mbox{ extends continuously to }t=0\quad\quad\\&\quad\quad\quad\quad\quad\quad\quad\quad\quad\quad\quad\quad\quad\quad\mbox{and }\Gamma_{\mathrm du(t,\cdot)+\beta(t)}\subset U_L\mbox{ for }t\in(0,T)\bigr\}.
\end{align*}
Then the operator $P$ is defined by
\[
P:\mathcal{D}\rightarrow C^{\infty}((0,T)\times L),\quad P(u)=\frac{\partial u}{\partial t}-\Theta(\Phi_L\circ(\mathrm du+\beta)).
\] 
If $u\in\mathcal{D}$, then $\Gamma_{\mathrm du(t,\cdot)+\beta(t)}\subset U_L$ for every $t\in(0,T)$, and the Lagrangian submanifold $\Phi_L\circ(\mathrm du(t,\cdot)+\beta(t)):L\rightarrow M$ is well defined for every $t\in(0,T)$. Hence $\Theta(\Phi_L\circ(\mathrm du(t,\cdot)+\beta(t)))$ is also well defined for every $t\in(0,T)$. 

We now consider the Cauchy problem
\begin{align}\label{CauchyProblem}
\begin{split}
&Pu(t,x)=0\;\;\;\;\mbox{for }(t,x)\in(0,T)\times L,\\ 
&u(0,x)=0\quad\;\;\;\mbox{for }x\in L.
\end{split}
\end{align}
If we are given a solution $u\in\mathcal{D}$ of the Cauchy problem \eq{CauchyProblem}, then we obtain a solution to the generalized Lagrangian mean curvature flow. In fact, the following proposition is easily checked using Proposition \ref{GeneralizedMeanCurvatureForm} and Lemma \ref{Var}. 

\begin{Prop}\label{IntegratingGMCF}
Let $u\in\mathcal{D}$ be a solution of \eq{CauchyProblem} and define a one-parameter family $\{F(t,\cdot)\}_{t\in(0,T)}$ of submanifolds of $M$ by
\begin{equation*}
F(t,\cdot):L\longrightarrow M,\quad F(t,\cdot)=\Phi_L\circ(\mathrm du(t,\cdot)+\beta(t)).
\end{equation*}
Then $\{F(t,\cdot)\}_{t\in(0,T)}$ is a one-parameter family of Lagrangian submanifolds, continuous up to $t=0$, which evolves by generalized Lagrangian mean curvature flow with initial condition $F_0:L\rightarrow M$.
\end{Prop}

From Proposition \ref{IntegratingGMCF} it follows that the short time existence problem of the generalized Lagrangian mean curvature flow for a compact Lagrangian submanifold $F_0:L\rightarrow M$ is equivalent to the short time existence of solutions to the Cauchy problem \eq{CauchyProblem}. Note in particular that \eq{CauchyProblem} is a fully nonlinear equation (in fact it is parabolic, as we will show below) of a scalar function only. Therefore we have ``integrated" the generalized Lagrangian mean curvature flow and got rid of the system of partial differential equations in \eq{GMCF}. Also note that if $F_0:L\rightarrow M$ has zero Maslov class, then we can choose $\beta_0=0$ in Proposition \ref{IntegratingGMCF}.

We now outline the proof of Theorem \ref{GMCFShortTime}. It suffices to show that the Cauchy problem \eq{CauchyProblem} admits a solution for a short time. Studying short time existence problems for scalar nonlinear parabolic equations is very similar to the study of elliptic deformation problems. In fact, once one can show that the operator $P$ is a smooth operator between certain Banach manifolds and that its linearization at the initial condition is an isomorphism, one can use the Inverse Function Theorem for Banach manifolds to show that for a short time there exists a solution with low regularity to the nonlinear equation. Thereafter, using standard regularity theory for parabolic equations, one can show that the solution is in fact is smooth. We will not enter the details here, but refer the reader to the author's DPhil thesis \cite[\S 5]{Behrndt1}, where the analysis of the Cauchy problem \eq{CauchyProblem} is carried out in full detail. Nevertheless we want to state the next lemma which gives a formula for the linearization of $P$ at the initial condition and also verifies that $P$ is in fact a nonlinear parabolic differential operator.

\begin{Lem}\label{LIN}
The linearization of the operator $P:\mathcal{D}\rightarrow C^{\infty}((0,T)\times L)$ at the initial condition is given by
\[
\mathrm dP(0)(u)=\frac{\partial u}{\partial t}-\Delta u+m\mathrm d\psi_{\beta}(\nabla u)+\mathrm d\theta_{\beta}(\hat{V}(\mathrm du)),
\]
where $u$ is a function on $(0,T)\times L$, $\psi_{\beta}=(\Phi_L\circ\beta)^*(\psi)$, $\theta_{\beta}=\theta(\Phi_L\circ\beta)$, $\hat{V}$ is defined in Lemma \ref{Var}, and the Laplace operator and $\nabla$ are computed using the time dependent Riemannian metric $(\Phi_L\circ\beta)^*(g)$ on $L$.
\end{Lem}

\noindent From Lemma \ref{LIN} we see that the linearization of $P$ is a second order parabolic differential operator and thus, as expected, $P$ is a nonlinear parabolic differential operator.

\section{Linear parabolic equations on Riemannian manifolds with conical singularities}\label{ANALYSIS}

In this chapter we review some results about linear parabolic equations on Riemannian manifolds with conical singularities. We follow closely the author's paper \cite{Behrndt2}, and in fact most parts of this section are taken from \cite{Behrndt2}. As mentioned in the end of \S\ref{GULLI} it is essential first to understand linear parabolic equations before studying short time existence problems for their nonlinear counterparts. We think that a good understanding of the material of this section is important in order to understand the short time existence of the generalized Lagrangian mean curvature flow with conical singularities modelled on stable special Lagrangian cones. We especially recommend the reader to take note of the notion of discrete asymptotics in \S\ref{DiscreteASYMPTOTICS} and how they are involved in the study of linear parabolic equations on Riemannian manifolds with conical singularities, see Theorem \ref{SobolevRegularityHeat} below. 

Let us first define the notion of Riemannian cones, Riemannian manifolds with conical singularities, and finally the notion of a radius function. We begin with the definition of Riemannian cones.

\begin{Def}\label{RiemannianCone}
Let $(\Sigma,h)$ be an $(m-1)$-dimensional compact and connected Riemannian manifold, $m\geq 1$. Let $C=(\Sigma\times(0,\infty))\sqcup\{0\}$ and $C'=\Sigma\times(0,\infty)$ and write a general point in $C'$ as $(\sigma,r)$. Define a Riemannian metric on $C'$ by $g=\mathrm dr^2+r^2h$. Then we say that $(C,g)$ is the Riemannian cone over $(\Sigma,h)$ with Riemannian cone metric $g$.
\end{Def}

Next we define Riemannian manifolds with conical singularities.

\begin{Def}\label{DefinitionConicalSingularities}
Let $(M,d)$ be a metric space, $x_1,\ldots,x_n$ distinct points in $M$, and denote $M'=M\backslash\{x_1,\ldots,x_n\}$. Assume that $M'$ has the structure of a smooth and connected $m$-dimensional manifold, and that we are given a Riemannian metric $g$ on $M'$ that induces the metric $d$ on $M'$. Then we say that $(M,g)$ is an $m$-dimensional Riemannian manifold with conical singularities $x_1,\ldots,x_n$, if the following hold.\vspace{0.17cm}

\begin{compactenum}
\item[{\rm(i)}] We are given $R>0$ such that $d(x_i,x_j)>2R$ for $1\leq i<j\leq n$ and compact and connected $(m-1)$-dimensional Riemannian manifolds $(\Sigma_i,h_i)$ for $i=1,\ldots,n$. Denote by $(C_i,g_i)$ the Riemannian cone over $(\Sigma_i,h_i)$ for $i=1,\ldots,n$. 
\item[{\rm(ii)}] For $i=1,\ldots,n$ denote $S_i=\{x\in M\;:\;0<d(x,x_i)<R\}$. Then there exist $\mu_i\in\mathbb{R}$ with $\mu_i>2$ and diffeomorphisms $\phi_i:\Sigma_i\times(0,R)\rightarrow S_i$, such that
\begin{equation*}
\left|\nabla^k(\phi_i^*(g)-g_i)\right|=O(r^{\mu_i-2-k})\quad\mbox{as }r\longrightarrow 0\;\mbox{for }k\in\mathbb{N}
\end{equation*}
and $i=1,\ldots,n$. Here $\nabla$ and $|\cdot|$ are computed using the Riemannian cone metric $g_i$ on $\Sigma_i\times(0,R)$ for $i=1,\ldots,n$.\vspace{0.17cm}
\end{compactenum}
Additionally, if $(M,d)$ is a compact metric space, then we say that $(M,g)$ is a compact Riemannian manifold with conical singularities. 
\end{Def}

Finally we introduce the notion of a radius function.

\begin{Def}
Let $(M,g)$ be a Riemannian manifold with conical singularities as in Definition \ref{DefinitionConicalSingularities}. A radius function on $M'$ is a smooth function $\rho:M'\rightarrow(0,1]$, such that $\rho\equiv 1$ on $M'\backslash\bigcup_{i=1}^nS_i$ and
\begin{equation*}
|\phi_i^*(\rho)-r|=O(r^{1+\varepsilon})\quad\mbox{as }r\longrightarrow 0
\end{equation*}
for some $\varepsilon>0$. Here $|\cdot|$ is computed using the Riemannian cone metric $g_i$ on $\Sigma_i\times(0,R)$ for $i=1,\ldots,n$. A radius function always exists.
\end{Def}

If $\rho$ is a radius function on $M'$ and $\boldsymbol\gamma=(\gamma_1,\ldots,\gamma_n)\in\mathbb{R}^n$, then we define a function $\rho^{\boldsymbol\gamma}$ on $M'$ as follows. On $S_i$ we set $\rho^{\boldsymbol\gamma}=\rho^{\gamma_i}$ for $i=1,\ldots,n$ and $\rho^{\boldsymbol\gamma}\equiv 1$ otherwise. Moreover, if $\boldsymbol\gamma,\boldsymbol\mu\in\mathbb{R}^n$, then we write $\boldsymbol\gamma\leq\boldsymbol\mu$ if $\gamma_i\leq\mu_i$ for $i=1,\ldots,n$, and $\boldsymbol\gamma<\boldsymbol\mu$ if $\gamma_i<\mu_i$ for $i=1,\ldots,n$. Finally, if $\boldsymbol\gamma\in\mathbb{R}^n$ and $a\in\mathbb{R}$, then we denote $\boldsymbol\gamma+a=(\gamma_1+a,\ldots,\gamma_n+a)\in\mathbb{R}^n$.

\subsection{Weighted Sobolev spaces}
In this and the following subsection we give a crash course in weighted Sobolev spaces and the Fredholm theory of the Laplace operator, or more generally of what we call operators of Laplace type, on Riemannian manifolds with conical singularities. For more details on the material presented here the reader should consult Joyce \cite{Joyce1}, Lockhart and McOwen \cite{LM}, the author \cite{Behrndt2}, and the references in these papers.

Throughout this subsection we denote by $(M,g)$ a compact $m$-dimensional Riemannian manifold with conical singularities as in Definition \ref{DefinitionConicalSingularities}. We first introduce weighted $C^k$-spaces. For $k\in\mathbb{N}$ we denote by $C^k_{\loc}(M)$ the space of $k$-times continuously differentiable functions $u:M'\rightarrow\mathbb{R}$ and we set $C^{\infty}(M')=\bigcap_{k\in\mathbb{N}}C^k_{\loc}(M')$, which is the space of smooth functions on $M'$. For $\boldsymbol\gamma\in\mathbb{R}^n$ we define the $C^k_{\boldsymbol\gamma}$-norm by
\[
\|u\|_{C^k_{\boldsymbol\gamma}}=\sum_{j=0}^k\sup_{x\in M'}|\rho(x)^{-\boldsymbol\gamma+j}\nabla^ju(x)|\quad\mbox{for }u\in C^k_{\loc}(M'),
\]
whenever it is finite. A different choice of radius function defines an equivalent norm. Note that $u\in C^k_{\loc}(M')$ has finite $C^k_{\boldsymbol\gamma}$-norm if and only if $\nabla^ju$ grows at most like $\rho^{\boldsymbol\gamma-j}$ for $j=0,\ldots,k$ as $\rho\rightarrow 0$. We define the weighted $C^k$-space $C^k_{\boldsymbol\gamma}(M')$ by
\[
C^k_{\boldsymbol\gamma}(M')=\bigl\{u\in C^k_{\loc}(M')\;:\;\|u\|_{C^k_{\boldsymbol\gamma}}<\infty\bigr\}.
\]
Then $C^k_{\boldsymbol\gamma}(M')$ is a Banach space. We also set $C^{\infty}_{\boldsymbol\gamma}(M')=\bigcap_{k\in\mathbb{N}}C^k_{\boldsymbol\gamma}(M')$. The space $C^{\infty}_{\boldsymbol\gamma}(M')$ is in general not a Banach space.

Next we define Sobolev spaces on $M'$. For a $k$-times weakly differentiable function $u:M'\rightarrow\mathbb{R}$ the $W^{k,p}$-norm is given by
\[
\|u\|_{W^{k,p}}=\left(\sum_{j=0}^k\int_{M'}|\nabla^ju|^p\;\mathrm dV_g\right)^{1/p},
\] 
whenever it is finite. Denote by $W^{k,p}_{\loc}(M')$ the space of $k$-times weakly differentiable functions on $M'$ that have locally a finite $W^{k,p}$-norm and define the Sobolev space $W^{k,p}(M')$ by
\[
W^{k,p}(M')=\left\{u\in W^{k,p}_{\loc}(M')\;:\;\|u\|_{W^{k,p}}<\infty\right\}.
\]
Then $W^{k,p}(M')$ is a Banach space. If $k=0$, then we write $L^p_{\loc}(M')$ and $L^p(M')$ instead of $W^{0,p}_{\loc}(M')$ and $W^{0,p}(M')$, respectively.

Finally we define weighted Sobolev spaces. For $k\in\mathbb{N}$, $p\in[1,\infty)$, and $\boldsymbol\gamma\in\mathbb{R}^n$ we define the $W^{k,p}_{\boldsymbol\gamma}$-norm by
\[
\|u\|_{W^{k,p}_{\boldsymbol\gamma}}=\left(\sum_{j=0}^k\int_{M'}|\rho^{-\boldsymbol\gamma+j}\nabla^ju|^p\rho^{-m}\;\mathrm dV_g\right)^{1/p}\quad\mbox{for }u\in W^{k,p}_{\loc}(M'),
\]
whenever it is finite. A different choice of radius function defines an equivalent norm. We define the weighted Sobolev space $W^{k,p}_{\boldsymbol\gamma}(M')$ by
\[
W^{k,p}_{\boldsymbol\gamma}(M')=\left\{u\in W^{k,p}_{\loc}(M')\;:\;\|u\|_{W^{k,p}_{\boldsymbol\gamma}}<\infty\right\}.
\]
Then $W^{k,p}_{\boldsymbol\gamma}(M')$ is a Banach space. If $k=0$, then we write $L^p_{\boldsymbol\gamma}(M')$ instead of $W^{0,p}_{\boldsymbol\gamma}(M')$. Note that $L^p(M')=L^p_{-m/p}(M')$ and that $C^{\infty}_{\cs}(M')$, the space of smooth functions on $M'$ with compact support, is dense in $W^{k,p}_{\boldsymbol\gamma}(M')$ for every $k\in\mathbb{N}$, $p\in[1,\infty)$, and $\boldsymbol\gamma\in\mathbb{R}^n$.

An important tool in the study of partial differential equations is the Sobolev Embedding Theorem, which gives embeddings of Sobolev spaces into different Sobolev spaces and $C^k$-spaces. The next theorem is a version of the Sobolev Embedding Theorem for weighted Sobolev spaces and weighted $C^k$-spaces.

\begin{Thm}\label{WeightedSobolevEmbedding}
Let $(M,g)$ be a compact $m$-dimensional Riemannian manifold with conical singularities as in Definition \ref{DefinitionConicalSingularities}. Let $k,l\in\mathbb{N}$, $p,q\in[1,\infty)$, and $\boldsymbol\gamma,\boldsymbol\delta\in\mathbb{R}^n$. Then the following hold.\vspace{0.17cm}
\begin{compactenum}
\item[{\rm(i)}] If $\frac{1}{p}\leq\frac{1}{q}+\frac{k-l}{m}$ and $\boldsymbol\gamma\geq\boldsymbol\delta$ then $W^{k,p}_{\boldsymbol\gamma}(M')$ embeds continuously into $W^{l,q}_{\boldsymbol\delta}(M')$ by inclusion.
\item[{\rm(ii)}] If $k-\frac{m}{p}>l$ and $\boldsymbol\gamma\geq\boldsymbol\delta$, then $W^{k,p}_{\boldsymbol\gamma}(M')$ embeds continuously into $C^{l}_{\boldsymbol\delta}(M')$ by inclusion.
\end{compactenum}
\end{Thm}

Another important result for the study of partial differential equations is the Rellich--Kondrakov Theorem, which states under which condition the embeddings in the Sobolev Embedding Theorem are compact. The next theorem is a version of the Rellich--Kondrakov Theorem for weighted H\"older and Sobolev spaces on compact Riemannian manifolds with conical singularities.

\begin{Thm}\label{WeightedRellichKondrakov}
Let $(M,g)$ be a compact $m$-dimensional Riemannian manifold with conical singularities as in Definition \ref{DefinitionConicalSingularities}. Let $k,l\in\mathbb{N}$, $p,q\in[1,\infty)$, and let $\boldsymbol\gamma,\boldsymbol\delta\in\mathbb{R}^n$. Then the following hold.\vspace{0.17cm}
\begin{compactenum}
\item[{\rm(i)}] If $\frac{1}{p}<\frac{1}{q}+\frac{k-l}{m}$ and $\boldsymbol\gamma>\boldsymbol\delta$, then the inclusion of $W^{k,p}_{\boldsymbol\gamma}(M')$ into $W^{l,q}_{\boldsymbol\delta}(M')$ is compact.
\item[{\rm(ii)}] If $k-\frac{m}{p}>l$ and $\boldsymbol\gamma>\boldsymbol\delta$, then the inclusion of $W^{k,p}_{\boldsymbol\gamma}(M')$ into $C^{l}_{\boldsymbol\delta}(M')$ is compact.
\end{compactenum}
\end{Thm}

\subsection{Operators of Laplace type on compact Riemannian manifolds with conical singularities}\label{GGG}

Before we can discuss the Fredholm theory for the Laplace operator, or more general for operators of Laplace type, on compact Riemannian manifolds with conical singularities, we need to study homogeneous harmonic functions on Riemannian cones.

Let $(\Sigma,h)$ be a compact and connected $(m-1)$-dimensional Riemannian manifold, $m\geq 1$, and let $(C,g)$ be the Riemannian cone over $(\Sigma,h)$ as in Definition \ref{RiemannianCone}. A function $u:C'\rightarrow\mathbb{R}$ is said to be homogeneous of order $\alpha$, if there exists a function $\varphi:\Sigma\rightarrow\mathbb{R}$, such that $u(\sigma,r)=r^{\alpha}\varphi(\sigma)$ for $(\sigma,r)\in C'$. A straightforward computation shows that the Laplace operator on $C'$ is given by $\Delta_{g}u=\partial_r^2u+(m-1)r^{-1}\partial_ru+r^{-2}\Delta_hu$, and the following lemma is easily verified.

\begin{Lem}\label{HarmonicFunctions}
A homogeneous function $u(\sigma,r)=r^{\alpha}\varphi(\sigma)$ of order $\alpha\in\mathbb{R}$ on $C'$ with $\varphi\in C^{\infty}(\Sigma)$ is harmonic if and only if $\Delta_h\varphi=-\alpha(\alpha+m-2)\varphi$.
\end{Lem}

Define
\begin{equation*}
\mathcal{D}_{\Sigma}=\{\alpha\in\mathbb{R}\;:\;-\alpha(\alpha+m-2)\mbox{ is an eigenvalue of }\Delta_h\}.
\end{equation*}
Then $\mathcal{D}_{\Sigma}$ is a discrete subset of $\mathbb{R}$ with no other accumulation points than $\pm\infty$. Moreover $\mathcal{D}_{\Sigma}\cap(2-m,0)=\emptyset$, since $\Delta_h$ is non-positive, and finally from Lemma \ref{HarmonicFunctions} it follows that $\mathcal{D}_{\Sigma}$ is the set of all $\alpha\in\mathbb{R}$ for which there exists a nonzero homogeneous harmonic function of order $\alpha$ on $C'$. Define a function
\[
m_{\Sigma}:\mathbb{R}\longrightarrow\mathbb{N},\quad m_{\Sigma}(\alpha)=\dim\ker(\Delta_h+\alpha(\alpha+m-2)).
\]
Then $m_\Sigma(\alpha)$ is the multiplicity of the eigenvalue $-\alpha(\alpha+m-2)$. Note that $m_{\Sigma}(\alpha)\neq 0$ if and only if $\alpha\notin\mathcal{D}_{\Sigma}$. Finally we define a function $M_{\Sigma}:\mathbb{R}\rightarrow\mathbb{Z}$ by
\begin{equation*}
M_{\Sigma}(\delta)=-\sum_{\alpha\in\mathcal{D}_{\Sigma}\cap(\delta,0)}m_{\Sigma}(\alpha)\;\mbox{if }\delta<0,\;M_{\Sigma}(\delta)=\sum_{\alpha\in\mathcal{D}_{\Sigma}\cap[0,\delta)}m_{\Sigma}(\alpha)\;\mbox{if }\delta\geq 0.
\end{equation*}
Then $M_{\Sigma}$ is a monotone increasing function that is discontinuous exactly on $\mathcal{D}_{\Sigma}$. As $\mathcal{D}_{\Sigma}\cap(2-m,0)=\emptyset$, we see that $M_{\Sigma}\equiv 0$ on $(2-m,0)$.
The set $\mathcal{D}_{\Sigma}$ and the function $M_{\Sigma}$ play an important r\^{o}le in the Fredholm theory for operators of Laplace type on compact Riemannian manifolds with conical singularities, see Theorem \ref{Fredholm} below. 

We now begin our review of the Fredholm theory for operators of Laplace type on compact Riemannian manifolds with conical singularities. From now on $(M,g)$ will denote a compact $m$-dimensional Riemannian manifold with conical singularities as in Definition \ref{DefinitionConicalSingularities}, and $\rho$ will be a radius function on $M'$. Operators of Laplace type are simply second order differential operators that are in leading order the Laplace operator. Before we give a precise definition of these operators let us have a closer look at what it means that a differential operator is the Laplace operator to leading order. For that consider the differential operator $D$ defined by $Du=\Delta_gu+g(X,\nabla u)+b\cdot u$, where $X$ is a vector field on $M'$ and $b, u$ are functions on $M'$. Let us assume that for some $\boldsymbol\delta\in\mathbb{R}^n$ we have $|\nabla^jX|=O(\rho^{\boldsymbol\delta-1-j})$ as $\rho\rightarrow 0$ for $j\in\mathbb{N}$ and $b\in C^{\infty}_{\boldsymbol\delta-2}(M')$. If $u\in C^{\infty}_{\boldsymbol\gamma}(M')$, then
\[
Du=\Delta_gu+g(X,\nabla u)+b\cdot u=O(\rho^{\boldsymbol\gamma-2})+O(\rho^{\boldsymbol\delta+\boldsymbol\gamma-2})
\]
Therefore, in general, the term $\Delta_gu$ dominates the lower order term $g(X,\nabla u)+b\cdot u$ near the singularity if and only if $\boldsymbol\delta>0$.

\begin{Def}\label{BERSERKER}
Let $D$ be a linear second order differential operator on $M'$. Then $D$ is said to be a differential operator of Laplace type if there exist $\boldsymbol\delta\in\mathbb{R}^n$ with $\boldsymbol\delta>0$, a vector field $X$ on $M'$ with $|\nabla^jX|=O(\rho^{\boldsymbol\delta-1-j})$ as $\rho\rightarrow 0$ for $j\in\mathbb{N}$, and a function $b\in C^{\infty}_{\boldsymbol\delta-2}(M')$, such that
\begin{equation}\label{LaplaceType}
Du=\Delta_gu+g(X,\nabla u)+b\cdot u\quad\mbox{for }u\in C^{\infty}(M').
\end{equation}
\end{Def}

Let $D$ be a differential operator of Laplace type as in \eq{LaplaceType} and define a first order differential operator $K$ by $Ku=g(X,\nabla u)+b\cdot u$. Then it easily follows from the Rellich--Kondrakov Theorem for weighted Sobolev spaces, Theorem \ref{WeightedRellichKondrakov}, that $K$ is a compact operator $W^{k,p}_{\boldsymbol\gamma}(M')\rightarrow W^{k-2,p}_{\boldsymbol\gamma-2}(M')$ for each $k\in\mathbb{N}$ with $k\geq 2$, $p\in(1,\infty)$, and $\boldsymbol\gamma\in\mathbb{R}^n$. Therefore operators of Laplace type, when mapping between weighted Sobolev spaces, differ from the Laplace operator only by a compact perturbation term. In particular it follows that the Laplace operator and operators of Laplace type essentially have the same Fredholm theory.

The next theorem is the main Fredholm theorem for operators of Laplace type on compact Riemannian manifolds with conical singularities, which can be easily deduced from the corresponding results for the Laplace operator discussed in \cite{Behrndt2} for instance. 

\begin{Thm}\label{Fredholm}
Let $(M,g)$ be a compact $m$-dimensional Riemannian manifold with conical singularities as in Definition \ref{DefinitionConicalSingularities}, $m\geq 3$, and $\boldsymbol\gamma\in\mathbb{R}^n$ and $D$ an operator of Laplace type. Let $k\in\mathbb{N}$ with $k\geq 2$ and $p\in(1,\infty)$. Then
\begin{equation}\label{p2}
D:W^{k,p}_{\boldsymbol\gamma}(M')\rightarrow W^{k-2,p}_{\boldsymbol\gamma-2}(M')
\end{equation}
is a Fredholm operator if and only if $\gamma_i\notin\mathcal{D}_{\Sigma_i}$ for $i=1,\ldots,n$. If $\gamma_i\notin\mathcal{D}_{\Sigma_i}$ for $i=1,\ldots,n$, then the Fredholm index of \eq{p2} is equal to $-\sum_{i=1}^nM_{\Sigma_i}(\gamma_i)$.
\end{Thm}

For our later study of linear parabolic equations on compact Riemannian manifolds with conical singularities we need to introduce some more notation. Denote by $(\Sigma,h)$ as above a compact and connected $(m-1)$-dimensional Riemannian manifold, $m\geq 1$, and let $(C,g)$ the Riemannian cone over $(\Sigma,h)$. Then we define
\[
\mathcal{E}_{\Sigma}=\mathcal{D}_{\Sigma}\cup\left\{\beta\in\mathbb{R}\;:\;\beta=\alpha+2k\mbox{ for }\alpha\in\mathcal{D}_{\Sigma},\;k\in\mathbb{N}\mbox{ with }\alpha\geq 0\mbox{ and }k\geq 1\right\}
\]
and a function $n_{\Sigma}:\mathbb{R}\longrightarrow\mathbb{N}$ by
\[
n_{\Sigma}(\beta)=m_{\Sigma}(\beta)+\sum_{k\geq 1,\;2k\leq\beta}m_{\Sigma}(\beta-2k).
\]
Clearly if $\beta\notin\mathcal{E}_{\Sigma}$, then $n_{\Sigma}(\beta)=0$. Also note that if $\beta<2$, then $n_{\Sigma}(\beta)=m_{\Sigma}(\beta)$. Moreover, if $\beta\in\mathcal{E}_{\Sigma}$, then $n_{\Sigma}(\beta)$ counts the multiplicity of the eigenvalues 
\[
-\beta(\beta+m-2),-(\beta-2)((\beta-2)+m-2),\ldots,-(\beta-2k)((\beta-2k)+m-2)
\]
for $2k\leq\beta$. Finally we define a function $N_{\Sigma}:\mathbb{R}\longrightarrow\mathbb{N}$ by
\begin{equation}\label{DEFN}
N_{\Sigma}(\delta)=-\sum_{\beta\in\mathcal{D}_{\Sigma}\cap(\delta,0)}n_{\Sigma}(\beta)\;\mbox{if }\delta<0,\;N_{\Sigma}(\delta)=\sum_{\beta\in\mathcal{D}_{\Sigma}\cap[0,\delta)}n_{\Sigma}(\beta)\;\mbox{if }\delta\geq 0.
\end{equation}
Then $N_{\Sigma}(\delta)=M_{\Sigma}(\delta)$ for $\delta\leq 2$ and
\begin{equation}\label{indexxx}
M_{\Sigma}(\delta)=N_{\Sigma}(\delta)-N_{\Sigma}(\delta-2)\quad\mbox{for }\delta\in\mathbb{R}\mbox{ with }\delta>2.
\end{equation}
The set $\mathcal{E}_{\Sigma}$ and the function $N_{\Sigma}$ play a similar r\^{o}le in the study of the heat equation on compact Riemannian manifolds with conical singularities as $\mathcal{D}_{\Sigma}$ and $M_{\Sigma}$ do in the study of the Laplace operator, see Theorem \ref{SobolevRegularityHeat} below.

\subsection{Discrete asymptotics for operators of Laplace type}\label{DiscreteASYMPTOTICS}

In this subsection we will introduce the notion of discrete asymptotics. It turns out that discrete asymptotics, as we will later see, are the reason why the conical singularities in the generalized Lagrangian mean curvature flow move around in the ambient space. From an analytical point of view discrete asymptotics are important because they enter into the study of the inhomogeneous heat equation on compact Riemannian manifolds with conical singularities. In fact it turns out to be necessary to introduce weighted Sobolev spaces with discrete asymptotics in order to prove maximal regularity of solutions to the inhomogeneous heat equation, see \S\ref{PIIEO} below.

We begin with the construction of the model space for the discrete asymptotics. Let $(\Sigma,h)$ be a compact and connected $(m-1)$-dimensional Riemannian manifold, $m\geq 1$, and let $(C,g)$ be the Riemannian cone over $(\Sigma,h)$. For $\gamma\in\mathbb{R}$ we denote
\[
H_{\gamma}(C')=\Span\left\{u=r^{\alpha}\varphi\;:\;0\leq\alpha<\gamma,\;\varphi\in C^{\infty}(\Sigma),\;u\mbox{ is harmonic}\right\},
\]
which is the space of homogeneous harmonic functions of order $\alpha$ with $0\leq\alpha<\gamma$. Then $\dim H_{\gamma}(C')=M_{\Sigma}(\gamma)$ for $\gamma\geq 2-m$, so $H_{\gamma}(C')$ is at least one dimensional for $\gamma>0$. We define a finite dimensional vector space $V_{\mathsf{P}_{\gamma}}(C')$ by
\[
V_{\mathsf{P}_{\gamma}}(C')=\Span\left\{v=r^{2k}u\;:\;k\in\mathbb{N},\;u=r^{\alpha}\varphi\in H_{\gamma}(C')\mbox{ and }\alpha+2k<\gamma\right\}.
\]
Note that the Laplace operator on $C'$ maps $V_{\mathsf{P}_{\gamma}}(C')\rightarrow V_{\mathsf{P}_{\gamma-2}}(C')$ for every $\gamma\in\mathbb{R}$ and is a nilpotent map $V_{\mathsf{P}_{\gamma}}(C')\rightarrow V_{\mathsf{P}_{\gamma}}(C')$. Also note that $\dim V_{\mathsf{P}_{\gamma}}(C')=N_{\Sigma}(\gamma)$ for $\gamma\geq 2-m$ and that $V_{\mathsf{P}_{\gamma}}(C')=H_{\gamma}(C')$ for $\gamma\leq 2$. The space $V_{\mathsf{P}_{\gamma}}(C')$ serves as the model space in the definition of discrete asymptotics on general Riemannian manifolds with conical singularities. 

The definition of discrete asymptotics on compact Riemannian manifolds with conical singularities is based on the following proposition, which can be found in \cite[Prop. 6.14]{Behrndt1}.

\begin{Prop}\label{Asymptotics}
Let $(M,g)$ be a compact $m$-dimensional Riemannian manifold with conical singularities as in Definition \ref{DefinitionConicalSingularities}, $m\geq 3$, $D$ an operator of Laplace type, and $\boldsymbol\gamma\in\mathbb{R}^n$. Then for every $\varepsilon>0$ there exists a linear map
\begin{equation*}\label{MAP}
\Psi_{\boldsymbol\gamma}^D:\bigoplus_{i=1}^nV_{\mathsf{P}_{\gamma_i}}(C_i')\longrightarrow C^{\infty}(M'),
\end{equation*}
such that the following hold.\vspace{0.17cm}
\begin{compactenum}
\item[{\rm(i)}] For every $v\in\bigoplus_{i=1}^nV_{\mathsf{P}_{\gamma_i}}(C_i')$ with $v=(v_1,\ldots,v_n)$ and $v_i=r^{\beta_i}\varphi_i$ where $\varphi_i\in C^{\infty}(\Sigma_i)$ for $i=1,\ldots,n$ we have that
\[
|\nabla^k(\phi_i^*(\Psi_{\boldsymbol\gamma}^D(v))-v_i)|=O(r^{\mu_i-2-\varepsilon+\beta_i-k})\quad\mbox{as }r\longrightarrow 0\mbox{ for }k\in\mathbb{N}
\]
and $i=1,\ldots,n$.
\item[{\rm(ii)}] For every $v\in\bigoplus_{i=1}^nV_{\mathsf{P}_{\gamma_i}}(C_i')$ with $v=(v_1,\ldots,v_n)$ we have that
\[
D(\Psi^D_{\boldsymbol\gamma}(v))-\sum_{i=0}^n\Psi^D_{\boldsymbol\gamma}(\Delta_{g_i}v_i)\in C^{\infty}_{\cs}(M').
\]
\end{compactenum}
\end{Prop}

Using Proposition \ref{Asymptotics} we can now define weighted $C^k$-spaces and Sobolev spaces with discrete asymptotics on compact Riemannian manifolds with conical singularities as follows. If $(M,g)$ is a compact $m$-dimensional Riemannian manifold with conical singularities, $m\geq 3$, then for $k\in\mathbb{N}$, $p\in[1,\infty)$, and $\boldsymbol\gamma\in\mathbb{R}^n$ we define
\[
C^k_{\boldsymbol\gamma,\mathsf{P}^D_{\boldsymbol\gamma}}(M')=C^k_{\boldsymbol\gamma}(M')\oplus\im\;\Psi^D_{\boldsymbol\gamma}\quad\mbox{and}\quad W^{k,p}_{\boldsymbol\gamma,\mathsf{P}^D_{\boldsymbol\gamma}}(M')=W^{k,p}_{\boldsymbol\gamma}(M')\oplus\im\;\Psi^D_{\boldsymbol\gamma}.
\] 
Then $C^k_{\boldsymbol\gamma,\mathsf{P}_{\boldsymbol\gamma}}(M')$ and $W^{k,p}_{\boldsymbol\gamma,\mathsf{P}_{\boldsymbol\gamma}}(M')$ are both Banach spaces, where the norm on the discrete asymptotics part is some finite dimensional norm. Note that the discrete asymptotics are trivial if $\boldsymbol\gamma\leq 0$, so that in this case the weighted spaces with discrete asymptotics are simply weighted spaces.

Using Theorem \ref{Fredholm}, $\dim V_{\mathsf{P}_{\gamma}}(C')=N_{\Sigma}(\gamma)$, where $N_{\Sigma}$ is defined in \eq{DEFN}, and equation \eq{indexxx} one can now prove the following result.

\begin{Prop}\label{Fredholm2}
Let $(M,g)$ be a compact $m$-dimensional Riemannian manifold with conical singularities as in Definition \ref{DefinitionConicalSingularities}, $m\geq 3$, and $D$ an operator of Laplace type. Let $k\in\mathbb{N}$ with $k\geq 2$, $p\in(1,\infty)$, and $\boldsymbol\gamma\in\mathbb{R}^n$ with $\boldsymbol\gamma>2-m$ and $\gamma_i\notin\mathcal{E}_{\Sigma_i}$ for $i=1,\ldots,n$. Then
\begin{equation*}
D:W^{k,p}_{\boldsymbol\gamma,\mathsf{P}_{\boldsymbol\gamma}^D}(M')\rightarrow W^{k-2,p}_{\boldsymbol\gamma-2,\mathsf{P}_{\boldsymbol\gamma-2}^D}(M')
\end{equation*}
is a Fredholm operator with index zero. In particular, if $b\equiv 0$ and $\boldsymbol\gamma>0$, then 
\[
D:\left\{u\in W^{k,p}_{\boldsymbol\gamma,\mathsf{P}_{\boldsymbol\gamma}^D}(M')\;:\;\mbox{$\int_{M'}u=0$}\right\}\longrightarrow\left\{u\in W^{k-2,p}_{\boldsymbol\gamma-2,\mathsf{P}_{\boldsymbol\gamma-2}^D}(M')\;:\;\mbox{$\int_{M'}u=0$}\right\}
\]
is an isomorphism.
\end{Prop}

\subsection{Weighted parabolic Sobolev spaces and linear parabolic equations of Laplace type}\label{PIIEO}

We now consider the following Cauchy problem
\begin{align}\label{HeatEquation}
\begin{split}
&\partial_tu(t,x)=Du(t,x)+f(t,x)\quad\mbox{for }(t,x)\in(0,T)\times M',\\
&u(0,x)=0\quad\quad\quad\quad\quad\quad\quad\quad\;\;\;\mbox{for } x\in M', 
\end{split}
\end{align}
where $f:(0,T)\times M'\rightarrow\mathbb{R}$ is a given function, $(M,g)$ is a compact Riemannian manifold with conical singularities, $T>0$, and $D$ is an operator of Laplace type. 

In order to state the main result about the existence and regularity of solutions to \eq{HeatEquation} correctly we first need to introduce weighted parabolic Sobolev spaces with discrete asymptotics.

We begin with the definition of Sobolev spaces of maps $u:I\rightarrow X$, where $I\subset\mathbb{R}$ is an open and bounded interval and $X$ is a Banach space. Let $k\in\mathbb{N}$ and $p\in[1,\infty)$. For a $k$-times weakly differentiable map $u:I\rightarrow X$ we define the $W^{k,p}$-norm by 
\[
\|u\|_{W^{k,p}}=\left(\sum_{j=0}^k\int_I\|\partial_t^ju(t)\|_X^p\;\mathrm dt\right)^{1/p},
\]
whenever it is finite. We denote by $W^{k,p}_{\loc}(I;X)$ the space of $k$-times weakly differentiable maps $u:I\rightarrow X$ with locally finite $W^{k,p}$-norm, and we define
\[
W^{k,p}(I;X)=\left\{u\in W^{k,p}_{\loc}(I;X)\;:\;\|u\|_{W^{k,p}}<\infty\right\}.
\] 
Then $W^{k,p}(I;X)$ is a Banach space. If $k=0$, then we write $L^p_{\loc}(I;X)$ and $L^p(I;X)$ instead of $W^{0,p}_{\loc}(I;X)$ and $W^{0,p}(I;X)$, respectively.

We can now define weighted parabolic Sobolev spaces as follows. Let $k,l\in\mathbb{N}$ with $2k\leq l$, $p\in[1,\infty)$, and $\boldsymbol\gamma\in\mathbb{R}^n$. The weighted parabolic Sobolev space $W^{k,l,p}_{\boldsymbol\gamma}(I\times M')$ is given by
\[
W^{k,l,p}_{\boldsymbol\gamma}(I\times M')=\bigcap_{j=0}^kW^{j,p}(I;W^{l-2j,p}_{\boldsymbol\gamma-2j}(M')).
\]
Then $W^{k,l,p}_{\boldsymbol\gamma}(I\times M')$ is a Banach space. Moreover, if $m\geq 3$, then we define the weighted parabolic Sobolev space $W^{k,l,p}_{\boldsymbol\gamma,\mathsf{P}^D_{\boldsymbol\gamma}}(I\times M')$ with discrete asymptotics by 
\[
W^{k,l,p}_{\boldsymbol\gamma,\mathsf{P}^D_{\boldsymbol\gamma}}(I\times M')=\bigcap_{j=0}^kW^{j,p}(I;W^{l-2j,p}_{\boldsymbol\gamma-2j,\mathsf{P}^D_{\boldsymbol\gamma-2j}}(M')).
\]
Clearly $W^{k,l,p}_{\boldsymbol\gamma,\mathsf{P}^D_{\boldsymbol\gamma}}(I\times M')$ is a Banach space.

In order to understand these rather complicated looking spaces let us consider the special case where $k=1$, $l=2$, and $\boldsymbol\gamma>2$ and let us see how a function $u\in W^{1,2,p}_{\boldsymbol\gamma,\mathsf{P}_{\boldsymbol\gamma}^D}(I\times M')$ behaves under the action of the heat operator. Loosely speaking the function $u$ is of the following form
\[
u(t,\cdot)=O(\rho^{\boldsymbol\gamma})+\mbox{discrete asymptotics of rate}<\boldsymbol\gamma
\]
for each $t\in I$. When we apply the operator $D$ to the function $u$ then, from the definition of the space $ W^{1,2,p}_{\boldsymbol\gamma,\mathsf{P}_{\boldsymbol\gamma}^D}(I\times M')$, it follows that $Du\in W^{0,0,p}_{\boldsymbol\gamma-2,\mathsf{P}_{\boldsymbol\gamma-2}^D}(I\times M')$. So by differentiating $u$ in the spatial direction using $L$ we lose two spatial derivatives, and therefore two rates of decay, and one time derivative. For parabolic Sobolev spaces it is natural that two spatial derivatives compare to one time derivative, so when we take two spatial derivatives we also lose one time derivative. This, first of all, explains why $Du\in W^{0,0,p}_{\boldsymbol\gamma-2,\mathsf{P}_{\boldsymbol\gamma-2}^D}(I\times M')$. Now, when studying the Cauchy problem \eq{HeatEquation}, $Du$ should have the same regularity as $\partial_tu$, and, as we can see from the definition of $W^{1,2,p}_{\boldsymbol\gamma,\mathsf{P}_{\boldsymbol\gamma}^D}(I\times M')$, we have in fact that $\partial_tu\in W^{0,0,p}_{\boldsymbol\gamma-2,\mathsf{P}_{\boldsymbol\gamma-2}^D}(I\times M')$. Thus when we take one time derivative of $u$, then we lose two spatial derivatives (as usual for parabolic equations), and therefore we also lose two rates of decay. Therefore, loosely speaking, we have that
\[
\partial_tu(t,\cdot),Du(t,\cdot)=O(\rho^{\boldsymbol\gamma-2})+\mbox{discrete asymptotics of rate}<\boldsymbol\gamma-2
\]
for $t\in I$. In particular we expect that if the function $f$ that we are given in the Cauchy problem \eq{HeatEquation} lies in $W^{0,0,p}_{\boldsymbol\gamma-2,\mathsf{P}_{\boldsymbol\gamma-2}^D}(I\times M')$, then the solution $u$ to the Cauchy problem \eq{HeatEquation}, if it exists, should lie in $W^{1,2,p}_{\boldsymbol\gamma,\mathsf{P}_{\boldsymbol\gamma}^D}(I\times M')$.

The following theorem is the main result about the existence and regularity of solutions to \eq{HeatEquation}. 

\begin{Thm}\label{SobolevRegularityHeat}
Let $(M,g)$ be a compact $m$-dimensional Riemannian manifold with conical singularities as in Definition \ref{DefinitionConicalSingularities}, $m\geq 3$. Let $T>0$, $k\in\mathbb{N}$ with $k\geq 2$, $p\in(1,\infty)$, and $\boldsymbol\gamma\in\mathbb{R}^n$ with $\boldsymbol\gamma>2-m$ and $\gamma_i\notin\mathcal{E}_{\Sigma_i}$ for $i=1,\ldots,n$. Given $f\in W^{0,k-2,p}_{\boldsymbol\gamma-2,\mathsf{P}^D_{\boldsymbol\gamma-2}}((0,T)\times M')$, then there exists a unique $u\in W^{1,k,p}_{\boldsymbol\gamma,\mathsf{P}^D_{\boldsymbol\gamma}}((0,T)\times M')$ solving the Cauchy problem \eq{HeatEquation}.
\end{Thm}

\noindent The proof of Theorem \ref{SobolevRegularityHeat} can be found in \cite[Thm. 4.8]{Behrndt2} for the case $D=\Delta_g$. 

We shortly explain why Theorem \ref{SobolevRegularityHeat} continues to hold for general operators of Laplace type. The proof of Theorem \ref{SobolevRegularityHeat} consists of three steps. The first step is to construct a fundamental solution, i.e. a function $H\in C^{\infty}((0,\infty)\times M'\times M')$ that solves the Cauchy problem
\begin{align}\label{EEEQW1}
\begin{split}
&\frac{\partial H}{\partial t}(t,x,y)=D_xH(t,x,y)\quad\mbox{for }(t,x,y)\in(0,T)\times M'\times M',\\
&H(0,x,y)=\delta_x(y)\;\quad\;\;\quad\quad\;\;\;\;\mbox{for }x\in M'.
\end{split}
\end{align}
This was done by Mooers in \cite{Mooers} for the case $D=\Delta_g$. Let $H_0$ be the solution to \eq{EEEQW1} in the case where $D=\Delta_g$, i.e. $H_0$ is the heat kernel. Let $D$ be a differential operator of Laplace type. Since $\Delta_g$ is the leading order term of $D$, $H_0$ is already a good approximation for a solution of \eq{EEEQW1}. Therefore $H_0$ satisfies
\begin{align*}
&\frac{\partial H_0}{\partial t}(t,x,y)=D_xH_0(t,x,y)+R_0(t,x,y)\quad\mbox{for }(t,x,y)\in(0,T)\times M'\times M',\\
&H_0(0,x,y)=\delta_x(y)\;\quad\;\;\quad\quad\quad\quad\quad\quad\quad\quad\;\;\;\mbox{for }x\in M',
\end{align*}
where the error term $R_0\in C^{\infty}((0,\infty)\times M'\times M')$ is in some sense a lower order term. Now one can follow the construction given by Mooers and construct a function $H_1\in C^{\infty}((0,\infty)\times M'\times M')$ which solves away the error term $R_0$. Then by defining $H=H_0+H_1$, one obtains a solution to \eq{EEEQW1} and the leading order term of $H$ is $H_0$. The second step in the proof of Theorem \ref{SobolevRegularityHeat} is to show that the fundamental solution $H$ satisfies certain estimates. In fact, these estimates are an immediate consequence of the way the fundamental solution is constructed. The third step is now completely analogous to the special case of the heat equation. Once the correct estimates for the fundamental solution are known, one can write down the solution of the Cauchy problem \eq{HeatEquation} explicitly as a convolution integral of $H$ with $f$ and then study the regularity of this convolution integral as in \cite[Thm. 4.8]{Behrndt2}.

\section{Lagrangian submanifolds with isolated conical singularities}\label{SUPERMANNOOO}
\subsection{Special Lagrangian cones}\label{SLCONESQQQ}
In this subsection we define special Lagrangian cones in $\mathbb{C}^m$ and introduce the notion of stable special Lagrangian cones. More about special Lagrangian cones can be found in Joyce \cite[\S 8]{JoyceBook} and in Ohnita \cite{Ohnita}.

We begin with the definition of special Lagrangian cones in $\mathbb{C}^m$.

\begin{Def}\label{DefSLcone}
Let $\iota_{\Sigma}:\Sigma\rightarrow\mathcal{S}^{2m-1}$ be a compact and connected $(m-1)$-dimensional submanifold of the $(2m-1)$-dimensional unit sphere $\mathcal{S}^{2m-1}$ in $\mathbb{R}^{2m}$. We identify $\Sigma$ with its image $\iota_{\Sigma}(\Sigma)\subset \mathcal{S}^{2m-1}$. Define $\iota:\Sigma\times[0,\infty)\rightarrow\mathbb{C}^m$ by $\iota(\sigma,r)=r\sigma$. Denote $C=(\Sigma\times(0,\infty))\sqcup\{0\}$, $C'=\Sigma\times(0,\infty)$ and identify $C$ and $C'$ with their images $\iota(C)$ and $\iota(C')$ under $\iota$ in $\mathbb{C}^m$. Then $C$ is a special Lagrangian cone with phase $e^{i\theta}$, if $\iota$ restricted to $\Sigma\times(0,\infty)$ is a special Lagrangian submanifold of $\mathbb{C}^m$ with phase $e^{i\theta}$ in the sense of Definition \ref{DefinitionSpecialLagrangian}.
\end{Def}

Let $C$ be a special Lagrangian cone in $\mathbb{C}^m$. In \S\ref{DiscreteASYMPTOTICS} we discussed homogeneous harmonic functions on Riemannian cones. On a special Lagrangian cone there is a special class of homogeneous harmonic functions, namely those induced by the moment maps of the automorphism group of $(\mathbb{C}^m,J',\omega',\Omega')$. The automorphism group of $(\mathbb{C}^m,\omega',g')$ is the Lie group $U(m)\ltimes\mathbb{C}^m$, where $\mathbb{C}^m$ acts by translations, and the automorphism group of $(\mathbb{C}^m,\omega',g',\Omega')$ is the Lie group $SU(m)\ltimes\mathbb{C}^m$. The Lie algebra $\mathfrak{u}(m)$ of $U(m)$ is the space of skew-adjoint complex linear transformations, and the Lie algebra $\mathfrak{su}(m)$ of $SU(m)$ is the space of the trace-free, skew-adjoint complex linear transformations. Note in particular that $\mathfrak{u}(m)=\mathfrak{su}(m)\oplus\mathfrak{u}(1)$.

Let $X=(A,v)\in\mathfrak{u}(m)\oplus\mathbb{C}^m$, with $A=(a_{ij})_{i,j=1,\ldots,m}$ and $v=(v_i)_{i=1,\ldots,m}$. Then $X$ acts as a vector field on $\mathbb{C}^m$. Since $U(m)\ltimes\mathbb{C}^m$ preserves $\omega'$, $X\;\lrcorner\;\omega'$ is a closed one-form on $\mathbb{C}^m$ and there exists a unique function $\mu_X:\mathbb{C}^m\rightarrow\mathbb{R}$, such that $\mathrm d\mu_X=X\;\lrcorner\;\omega'$ and $\mu_X(0)=0$. Indeed, if $X=(A,v)\in\mathfrak{u}(m)\oplus\mathbb{C}^m$, then $\mu_X$ is given by
\begin{equation*}
\mu_X=\frac{i}{2}\sum_{i,j=1}^ma_{ij}z_i\bar{z}_j+\frac{i}{2}\sum_{i=1}^m(v_i\bar{z}_i-\bar{v}_iz_i).
\end{equation*}
Moreover, since $a_{ij}=-\bar{a}_{ji}$ for $i,j=1,\ldots,n$, we see that $\mu_X$ is a real quadratic polynomial. We call $\mu_X$ a moment map for $X$. For $X=(A,v,c)\in\mathfrak{u}(m)\oplus\mathbb{C}^m\oplus\mathbb{R}$ we define $\mu_X:\mathbb{C}^m\rightarrow\mathbb{R}$ by requiring that 
\begin{equation*}
\mathrm d\mu_X=X\;\lrcorner\;\omega'\quad\mbox{and}\quad\mu_X(0)=c.
\end{equation*} 

A proof of the following proposition is given in Joyce \cite[Prop. 3.5]{Joyce2}.

\begin{Prop}\label{SLConeHarmonics}
Let $C$ be a special Lagrangian cone in $\mathbb{C}^m$ as in Definition \ref{DefSLcone} and let $G$ be the maximal Lie subgroup of $SU(m)$ that preserves $C$. Then the following hold.\vspace{0.17cm}
\begin{compactenum}
\item[{\rm(i)}] Let $X\in\mathfrak{su}(m)$. Then $\iota^*(\mu_X)$ is a homogeneous harmonic function of order two on $C'$. Consequently the space of homogeneous harmonic functions of order two on $C'$ is at least of dimension $m^2-1-\dim G$.
\item[{\rm(ii)}] Let $X\in\mathbb{C}^m$. Then $\iota^*(\mu_X)$ is a homogeneous harmonic function of order one on $C'$. Consequently the space of homogeneous harmonic functions of order one on $C'$ is at least of dimension $2m$.
\end{compactenum}
\end{Prop}

\noindent Also note that if $C$ is a special Lagrangian cone in $\mathbb{C}^m$ and $X\in\mathfrak{u}(1)$, then $\iota^*(\mu_X)=cr^2$ for some $c\in\mathbb{R}$.

Using Proposition \ref{SLConeHarmonics} we can define the stability index of a special Lagrangian cone in $\mathbb{C}^m$ and the notion of stable special Lagrangian cones as introduced by Joyce in \cite[Def. 3.6]{Joyce2}.

\begin{Def}\label{STABILITY}
Let $C$ be a special Lagrangian cone in $\mathbb{C}^m$ as in Definition \ref{DefSLcone} and let $G$ be the maximal Lie subgroup of $SU(m)$ that preserves $C$. Then the stability index of $C$ is the integer
\[
\sindex(C)=M_{\Sigma}(2)-m^2-2m+\dim G,
\]
where $M_{\Sigma}$ is defined as in \S\ref{DiscreteASYMPTOTICS}. From Proposition \ref{SLConeHarmonics} it follows that the stability index of a special Lagrangian cone is a non-negative integer. We say that a special Lagrangian cone $C$ in $\mathbb{C}^m$ is stable if $\sindex(C)=0$.
\end{Def}

\noindent Note that if $C$ is a stable special Lagrangian cone as in Definition \ref{DefSLcone}, then the only homogeneous harmonic functions on $C'$ with rate $\alpha$, where $0\leq\alpha\leq 2$, are those induced by the $SU(m)\ltimes\mathbb{C}^m$-moment maps. In particular, when $\gamma>2$ is sufficiently small, then $V_{\mathsf{P}_{\gamma}}(C')$ is spanned by the $SU(m)\ltimes\mathbb{C}^m$-moment maps and the functions $r^2$ and $1$.

Examples of special Lagrangian cones can be found in Joyce \cite[\S 8.3.2]{JoyceBook}. Examples of stable special Lagrangian cones, however, are hard to find and there are only a few examples known. The simplest example of a stable special Lagrangian cone is the Riemannian cone in $\mathbb{C}^3$ over $T^2$ with its standard metric. In this case $C$ is given by
\[
C=\left\{\left(re^{i\phi_1},re^{i\phi_2},re^{i(\phi_1-\phi_2)}\right)\;:\;r\in[0,\infty),\;\phi_1,\phi_2\in[0,2\pi)\right\}\subset\mathbb{C}^3
\]
together with the Riemannian metric induced by the Euclidean metric on $\mathbb{C}^3$. Some other examples of stable special Lagrangian cones can be found in Ohnita \cite{Ohnita}.

\subsection{Lagrangian submanifolds with isolated conical singularities}
In this subsection we define Lagrangian submanifolds with isolated conical singularities in almost Calabi--Yau manifolds. Before we define Lagrangian submanifolds with isolated conical singularities, we have to introduce the notion of manifolds with ends.

\begin{Def}\label{ManifoldsEnds}
Let $L$ be an open and connected $m$-dimensional manifold with $m\geq 1$. Assume that we are given a compact $m$-dimensional submanifold $K\subset L$ with boundary, such that $L\backslash K$ has a finite number of pairwise disjoint, open, and connected components $S_1,\ldots,S_n$. Then $L$ is a manifold with ends $S_1,\ldots,S_n$ if the following holds. There exist compact and connected $(m-1)$-dimensional manifolds $\Sigma_1,\ldots,\Sigma_n$, a constant $R>0$, and diffeomorphisms $\phi_i:\Sigma_i\times(0,R)\rightarrow S_i$ for $i=1,\ldots,n$. We say that $S_1,\ldots,S_n$ are the ends of $L$ and that $\Sigma_i$ is the link of $S_i$. Note that the boundary of $K$ is diffeomorphic to $\bigsqcup_{i=1}^n\Sigma_i$.
\end{Def}

Next we define Lagrangian submanifolds with isolated conical singularities following Joyce \cite[Def. 3.6]{Joyce1}.

\begin{Def}\label{DefLagrangianConicalSingularities}
Let $(M,J,\omega,\Omega)$ be an $m$-dimensional almost Calabi--Yau manifold and define $\psi\in C^{\infty}(M)$ as in \eq{DefinitionPsi}. Let $x_1,\ldots,x_n\in M$ be distinct points in $M$, $C_1,\ldots,C_n$ special Lagrangian cones in $\mathbb{C}^m$ as in Definition \ref{DefSLcone} with embeddings $\iota_i:\Sigma_i\times(0,\infty)\rightarrow\mathbb{C}^m$ for $i=1,\ldots,n$, and finally let $L$ be an $m$-dimensional manifold with ends as in Definition \ref{ManifoldsEnds}. Then a Lagrangian submanifold $F:L\rightarrow M$ is a Lagrangian submanifold with isolated conical singularities at $x_1,\ldots,x_n$ modelled on the special Lagrangian cones $C_1,\ldots,C_n$, if the following holds.

We are given isomorphisms $A_i:\mathbb{C}^m\rightarrow T_{x_i}M$ for $i=1,\ldots,n$ with $A_i^*(\omega)=\omega'$ and $A_i^*(\Omega)=e^{i\theta_i+m\psi(x_i)}\Omega'$ for some $\theta_i\in\mathbb{R}$ and $i=1,\ldots,n$. Then by Theorem \ref{DarbouxTheorem} there exist $R>0$ and embeddings $\Upsilon_i:B_R\rightarrow M$ with $\Upsilon_i(0)=x_i$, $\Upsilon_i^*(\omega)=\omega'$, and $\mathrm d\Upsilon_i(0)=A_i$ for $i=1,\ldots,n$. Making $R>0$ smaller if necessary we can assume that $\Upsilon_1(B_R),\ldots,\Upsilon_n(B_R)$ are pairwise disjoint in $M$. Then there should exist diffeomorphisms $\phi_i:\Sigma_i\times(0,R)\rightarrow S_i$ for $i=1,\ldots,n$, such that $F\circ\phi_i$ maps $\Sigma_i\times(0,R)\rightarrow\Upsilon_i(B_R)$ for $i=1,\ldots,n$, and there should exist $\boldsymbol\nu\in\mathbb{R}^n$ with $2<\boldsymbol\nu<3$, such that
\begin{equation}\label{ConeCondition}
\left|\nabla^k(\Upsilon_i^{-1}\circ F\circ\phi_i-\iota_i)\right|=O(r^{\nu_i-1-k})\quad\mbox{as }r\longrightarrow 0\mbox{ for }k\in\mathbb{N}.
\end{equation}
Here $\nabla$ and $|\cdot|$ are computed using the Riemannian cone metric $\iota^*_i(g')$ on $\Sigma_i\times(0,R)$. A Lagrangian submanifold $F:L\rightarrow M$ with isolated conical singularities modelled on special Lagrangian cones $C_1,\ldots,C_n$ is said to have stable conical singularities, if $C_1,\ldots,C_n$ are stable special Lagrangian cones in $\mathbb{C}^m$.
\end{Def}

We have chosen $2<\boldsymbol\nu<3$ in Definition \ref{DefLagrangianConicalSingularities} for the following reason. We need $\nu_i>2$ or otherwise \eq{ConeCondition} does not force the submanifold $F:L\rightarrow M$ to approach the cone $A_i(C_i)$ in $T_{x_i}M$ near $x_i$ for $i=1,\ldots,n$. Moreover $\nu_i<3$ guarantees that the definition is independent of the choice of $\Upsilon_i$. Indeed, if we are given a different embedding $\tilde{\Upsilon}_i:B_R\rightarrow M$ with $\tilde{\Upsilon}_i(0)=x_i$, $\tilde{\Upsilon}_i^*(\omega)=\omega'$, and $\mathrm d\tilde{\Upsilon}_i(0)=A_i$, then $\Upsilon_i-\tilde{\Upsilon}_i=O(r^2)$ on $B_R$ by Taylor's Theorem. Therefore, since $\nu_i<3$, it follows that \eq{ConeCondition} holds with $\Upsilon_i$ replaced by $\tilde{\Upsilon}_i$.

If $F:L\rightarrow M$ is a Lagrangian submanifold with isolated conical singularities as in Definition \ref{DefLagrangianConicalSingularities}, then \eq{ConeCondition} implies that $L\sqcup\{x_1,\ldots,x_n\}$ together with the Riemannian metric $F^*(g)$ is a Riemannian manifold with conical singularities $x_1,\ldots,x_n$ in the sense of Definition \ref{DefinitionConicalSingularities}.

Let $F:L\rightarrow M$ be a Lagrangian submanifold with isolated conical singularities $x_1,\ldots,x_n$ modelled on stable special Lagrangian cones $C_1,\ldots,C_n$, and let $D$ be an operator of Laplace type on $L$ (with respect to the induced Riemannian metric $F^*(g)$). We now want to study how solutions of the Cauchy problem \eq{HeatEquation} on $L$ look like. Let $k\in\mathbb{N}$ with $k\geq 2$, $p\in(1,\infty)$, $\boldsymbol\gamma\in\mathbb{R}^n$ with $\boldsymbol\gamma>2$ and $(2,\gamma_i]\cap\mathcal{E}_{\Sigma_i}=\emptyset$, and $T>0$. If $f\in W^{0,k-2,p}_{\boldsymbol\gamma-2,\mathsf{P}_{\boldsymbol\gamma-2}^D}((0,T)\times L)$, then by Theorem \ref{SobolevRegularityHeat}, there exists a unique $u\in W^{1,k,p}_{\boldsymbol\gamma,\mathsf{P}_{\boldsymbol\gamma}^D}((0,T)\times L)$ solving the Cauchy problem \eq{HeatEquation} and, loosely speaking, $u(t,\cdot)$ is of the form
\begin{align*}
u(t,\cdot)&=O(\rho^{\boldsymbol\gamma})+\rho^2+\mbox{homogeneous harmonic functions with rate}<\boldsymbol\gamma,
\end{align*}
where $\rho$ is a radius function on $L$. Since the the conical singularities are modelled on stable special Lagrangian cones, we have in fact that
\begin{align*}
u(t,\cdot)&=O(\rho^{\boldsymbol\gamma})+U(1)-\mbox{moment maps}+SU(m)-\mbox{moment maps}\\&\quad\quad\quad\quad\quad\quad\quad\quad+\mathbb{C}^m-\mbox{moment maps}+\mbox{constants}
\end{align*}
for $t\in(0,T)$. Therefore for $t\in(0,T)$
\[
u(t,\cdot)=O(\rho^{\boldsymbol\gamma})+\mbox{geometric motions of model cones}+\mbox{constants}.
\]

The following proposition shows that the Maslov class of a Lagrangian submanifold with isolated conical singularities is an element of $H^1_{cs}(L,\mathbb{R})$, the first compactly supported de Rham cohomology group of $L$. We use this result, when we set up the short time existence problem for the Lagrangian mean curvature flow with isolated conical singularities.

\begin{Prop}\label{PPP}
Let $F:L\rightarrow M$ be a Lagrangian submanifold with isolated conical singularities as in Definition \ref{DefLagrangianConicalSingularities}. Then the Maslov class $\mu_F$ of $F:L\rightarrow M$ may be defined as an element of $H^1_{cs}(L,\mathbb{R})$.
\end{Prop}

\subsection{Lagrangian neighbourhood theorems}\label{Neighbourhoods}
In this section we discuss various Lagrangian neighbourhood theorems for Lagrangian submanifolds with isolated conical singularities. Our discussion more or less follows Joyce \cite[\S 4]{Joyce1}.

We begin with a Lagrangian neighbourhood theorem for special Lagrangian cones in $\mathbb{C}^m$.

\begin{Thm}\label{LagrangianNeighbourhoodCone}
Let $\iota:C\rightarrow\mathbb{C}^m$ be a special Lagrangian cone as in Definition \ref{DefSLcone}. For $\sigma\in\Sigma$, $\tau\in T^*_{\sigma}\Sigma$, and $\varrho\in\mathbb{R}$ we denote by $(\sigma,r,\tau,\varrho)$ the point $\tau+\varrho\mathrm dr$ in $T^*_{(\sigma,r)}(\Sigma\times(0,\infty))$. For some sufficiently small $\zeta>0$ define an open neighbourhood $U_C$ of the zero section in $T^*(\Sigma\times(0,\infty))$ by
\[
U_C=\left\{(\sigma,r,\tau,\varrho)\in T^*(\Sigma\times(0,\infty))\;:\;|(\tau,\varrho)|<\zeta r\right\}.
\]
Then there exists a Lagrangian neighbourhood $\Phi_C:U_C\rightarrow\mathbb{C}^m$ for $\iota:C\rightarrow\mathbb{C}^m$.
\end{Thm}

Note that if $\iota:C\rightarrow\mathbb{C}^m$ is a special Lagrangian cone and $(A,v)\in U(m)\ltimes\mathbb{C}^m$, then $\tilde{\iota}=(A,v)\circ\iota:C\rightarrow\mathbb{C}^m$ is a special Lagrangian cone that is obtained by rotating $\iota:C\rightarrow\mathbb{C}^m$ using $A$ and translating it using $v$. In particular, when $\Phi_C:U_C\rightarrow\mathbb{C}^m$ is a Lagrangian neighbourhood for $\iota:C\rightarrow\mathbb{C}^m$, then $\tilde{\Phi}_C=(A,v)\circ\Phi_C:U_C\rightarrow\mathbb{C}^m$ is a Lagrangian neighbourhood for $\tilde{\iota}:C\rightarrow\mathbb{C}^m$.

Next we discuss a Lagrangian neighbourhood theorem for Lagrangian submanifolds with isolated conical singularities. Let us first assume that $F:L\rightarrow\mathbb{C}^m$ is a Lagrangian submanifold with isolated conical singularities in $\mathbb{C}^m$ as in Definition \ref{DefLagrangianConicalSingularities}, and assume additionally that near each conical singularity $F:L\rightarrow\mathbb{C}^m$ is an exact cone, i.e. $F(x)=\iota_i(\sigma,r)$ for $x=\phi_i(\sigma,r)\in S_i$ and $i=1,\ldots,n$. Then one can construct a Lagrangian neighbourhood for $F:L\rightarrow\mathbb{C}^m$ simply by using the Lagrangian neighbourhoods given by Theorem \ref{LagrangianNeighbourhoodCone} near the conical singularities and a Lagrangian neighbourhood as given by Theorem \ref{LagrangianNeighbourhood} away from the conical singularities and glueing them appropriately together in a region away from the conical singularities. We thus have the following theorem.

\begin{Thm}\label{UYA}
Let $F:L\rightarrow\mathbb{C}^m$ be a Lagrangian submanifold with isolated conical singularities $x_1,\ldots,x_n$ and cones $C_1,\ldots,C_n$ as in Definition \ref{DefLagrangianConicalSingularities}, and assume that $F(x)=\iota_i(\sigma,r)$ for $x=\phi_i(\sigma,r)\in S_i$ and $i=1,\ldots,n$, i.e. $F:L\rightarrow\mathbb{C}^m$ is an exact cone near each singularity. Let $\Phi_{C_i}:U_{C_i}\rightarrow\mathbb{C}^m$ be a Lagrangian neighbourhood for $\iota_i:C_i\rightarrow\mathbb{C}^m$ for $i=1,\ldots,n$. Then there exists an open neighbourhood $U_L$ of the zero section in $T^*L$ with $U_L\cap T^*S_i=\mathrm d\phi_i(U_{C_i})$ for $i=1,\ldots,n$ and a Lagrangian neighbourhood $\Phi_L:U_L\rightarrow\mathbb{C}^m$ for $F:L\rightarrow\mathbb{C}^m$ such that $\Phi_L\circ\mathrm d\phi_i=\Phi_{C_i}$ on $T^*(\Sigma_i\times(0,R))$ for $i=1,\ldots,n$.
\end{Thm}

Now let us discuss the general case, when $F:L\rightarrow M$ is a Lagrangian submanifold in an almost Calabi--Yau manifold with conical singularities as in Definition \ref{DefLagrangianConicalSingularities}. Near each conical singularity $x_i$ we are given Darboux coordinates $\Upsilon_i:B_R\rightarrow M$ with $\Upsilon_i(x_i)=0$ from Definition \ref{DefLagrangianConicalSingularities}. Then $\Upsilon_i^{-1}\circ F:S_i\rightarrow B_R\subset\mathbb{C}^m$ has a conical singularity at the origin that is modelled on the special Lagrangian cone $\iota_i:C_i\rightarrow\mathbb{C}^m$. Let $\Phi_{C_i}:U_{C_i}\rightarrow\mathbb{C}^m$ be a Lagrangian neighbourhood for $\iota_i:C_i\rightarrow\mathbb{C}^m$ as given by Theorem \ref{LagrangianNeighbourhoodCone}. Since $\Upsilon_i^{-1}\circ F:S_i\rightarrow B_R$ is asymptotic to $\iota_i:C_i\rightarrow\mathbb{C}^m$ with rate $\nu_i\in(2,3)$, it follows that the image of $\Upsilon_i^{-1}\circ F:S_i\rightarrow B_R$ will lie, at least near the conical singularity, in the set $\Phi_{C_i}(U_{C_i})$ and that we can write $\Upsilon_i^{-1}\circ F=\Phi_{C_i}\circ\mathrm da_i$ for some function $a_i\in C^{\infty}(\Sigma_i\times(0,T))$ with $|\nabla^ja_i|=O(r^{\nu_i-j})$ as $r\rightarrow 0$ for $j\in\mathbb{N}$. Following the same ideas as in the discussion prior to Theorem \ref{UYA} we then obtain the following Lagrangian neighbourhood theorem.

\begin{Thm}\label{LagrangianNeighbourhoodConicalSingularities}
Let $F:L\rightarrow\mathbb{C}^m$ be a Lagrangian submanifold in an almost Calabi--Yau manifold with conical singularities $x_1,\ldots,x_n$ and cones $C_1,\ldots,C_n$ as in Definition \ref{DefLagrangianConicalSingularities}. Let $\Phi_{C_i}:U_{C_i}\rightarrow\mathbb{C}^m$ be a Lagrangian neighbourhood for $\iota_i:C_i\rightarrow\mathbb{C}^m$ for $i=1,\ldots,n$. Then there exists an open neighbourhood $U_L$ of the zero section in $T^*L$ with $U_L\cap T^*S_i=\mathrm d\phi_i(U_{C_i})$ for $i=1,\ldots,n$, a function $a\in C^{\infty}_{\boldsymbol\nu}(L)$ with $\Gamma_{\mathrm da}\subset U_L$, and a Lagrangian neighbourhood $\Phi_L:U_L\rightarrow M$ for $F:L\rightarrow M$, such that $\Phi_L\circ\mathrm d\phi_i=\Phi_{C_i}\circ\mathrm da_i$ on $T^*(\Sigma_i\times(0,R))$ for $i=1,\ldots,n$, where $a_i=\phi_i^*(a)$ for $i=1,\ldots,n$.
\end{Thm}

So far we have only discussed Lagrangian neighbourhoods for a single Lagrangian submanifold with isolated conical singularities, but for later purposes we need to extend Theorem \ref{LagrangianNeighbourhoodConicalSingularities} to families of Lagrangian submanifolds with conical singularities. In order to do this we will follow the same ideas as in Joyce \cite[\S 5.1]{Joyce2}. Again we fix a Lagrangian submanifold $F:L\rightarrow M$ with conical singularities $x_1,\ldots,x_n$, model cones $C_1,\ldots,C_n$, and isomorphisms $A_i:\mathbb{C}^m\rightarrow T_{x_i}M$ for $i=1,\ldots,n$ as in Definition \ref{DefLagrangianConicalSingularities}. We define a fibre bundle $\mathcal{A}$ over $M$ by
\begin{gather*}
\mathcal{A}=\bigl\{(x,A)\;:\;x\in M,\;A:\mathbb{C}^m\longrightarrow T_xM,\;\;\;\;\;\;\;\;\;\;\;\;\;\;\;\;\;\;\;\;\;\;\;\;\;\;\;\;\;\;\;\;\;\;\;\;\;\;\;\;\\\;\;\;\;\;\;\;\;\;\;\;\;\;\;\;\;\;\;\;\;\;\;\;A^*(\omega)=\omega',\;A^*(\Omega)=e^{i\theta+m\psi(x)}\Omega'\mbox{ for some }\theta\in\mathbb{R}\bigr\}.
\end{gather*}
Then $B\in U(m)$ acts on $(x,A)\in\mathcal{A}_x$ by $B(x,A)=(x,A\circ B)$. This action of $U(m)$ is free and transitive on the fibres of $\mathcal{A}$ and thus $\mathcal{A}$ is a principal $U(m)$-bundle over $M$ with $\dim\mathcal{A}=m^2+2m$.

Let $G_i$ be the maximal Lie subgroup of $SU(m)$ that preserves $C_i$ for $i=1,\ldots,n$. If $(x_i,A_i)$ and $(x_i,\hat{A}_i)$ lie in the same $G_i$-orbit, then they define equivalent choices for $(x_i,A_i)$ in Definition \ref{DefLagrangianConicalSingularities}. To avoid this let $\mathcal{E}_i$ be a small open ball of dimension $\dim\mathcal{A}-\dim G_i$ containing $(x_i,A_i)$, which is transverse to the orbits of $G_i$ for $i=1,\ldots,n$. Then $G_i\cdot\mathcal{E}_i$ is open in $\mathcal{A}$. We set $\mathcal{E}=\mathcal{E}_1\times\ldots\times\mathcal{E}_n$, denote $e_0=(x_1,A_1,\ldots,x_n,A_n)$, and we equip $\mathcal{E}$ with the Riemannian metric induced by the Riemannian metric on $M$. Then $\mathcal{E}$ parametrizes all nearby alternative choices for $(x_i,A_i)$ in Definition \ref{DefLagrangianConicalSingularities}. Note that $\dim\mathcal{E}_i=m^2+2m-\dim G_i$ for $i=1,\ldots,n$ and $\dim\mathcal{E}=n(m^2+2m)-\sum_{i=1}^n\dim G_i$.

In the next lemma we prove the existence of a specific family of symplectomorphisms of $M$ that is parametrized by $e\in\mathcal{E}$ and which will allow us to construct a family of Lagrangian neighbourhoods for Lagrangian submanifolds with isolated conical singularities that are close to $F:L\rightarrow M$.

\begin{Lem}\label{LEMUU}
After making $\mathcal{E}$ smaller if necessary there exists a family $\{\Psi^e_M\}_{e\in\mathcal{E}}$ of smooth diffeomorphisms $\Psi^e_M:M\rightarrow M$, which depends smoothly on $e\in\mathcal{E}$, such that\vspace{0.17cm}
\begin{compactenum}
\item[{\rm(i)}] $\Psi^{e_0}_M$ is the identity on $M$,
\item[{\rm(ii)}] $\Psi^e_M$ is the identity on $M\backslash\bigcup_{i=1}^n\Upsilon_i(B_{R/2})$ for $e\in\mathcal{E}$,
\item[{\rm(iii)}] $(\Psi^e_M)^*(\omega)=\omega$ for $e\in\mathcal{E}$,
\item[{\rm(iv)}] $(\Psi^e_M\circ\Upsilon_i)(0)=\hat{x}_i$ and $\mathrm d(\Psi^e_M\circ\Upsilon_i)(0)=\hat{A}_i$ for $i=1,\ldots,n$ and $e\in\mathcal{E}$ with $e=(\hat{x}_1,\hat{A}_1,\ldots,\hat{x}_n,\hat{A}_n)$.
\end{compactenum}
\end{Lem}

\begin{proof}
It is useful to understand the proof of Lemma \ref{LEMUU} as it explains how Lagrangian submanifolds with isolated conical singularities can be deformed.

In order to prove Lemma \ref{LEMUU} we first construct families $\{\Psi^e_i\}_{e\in\mathcal{E}}$ of diffeomorphisms $\Psi^e_i:B_R\rightarrow B_R$ for $i=1,\ldots,n$, which depend smoothly on $e\in\mathcal{E}$ and satisfy\vspace{0.17cm}
\begin{compactenum}
\item[{\rm(a)}] $\Psi^{e_0}_i$ is the identity on $B_R$ for $i=1,\ldots,n$,
\item[{\rm(b)}] $\Psi^e_i$ is the identity on $B_R\backslash B_{R/2}$ for $e\in\mathcal{E}$ and $i=1,\ldots,n$,
\item[{\rm(c)}] $(\Psi^e_i)^*(\omega')=\omega'$ for $e\in\mathcal{E}$ and $i=1,\ldots,n$,
\item[{\rm(d)}] $(\Upsilon_i\circ\Psi^e_i)(0)=\hat{x}_i$ and $\mathrm d(\Upsilon_i\circ\Psi^e_i)(0)=\hat{A}_i$ for $i=1,\ldots,n$ and $e\in\mathcal{E}$ with $e=(\hat{x}_1,\hat{A}_1,\ldots,\hat{x}_n,\hat{A}_n)$.\vspace{0,2cm}
\end{compactenum}

Let $e=(\hat{x}_1,\hat{A}_1,\ldots,\hat{x}_n,\hat{A}_n)\in\mathcal{E}$. By making $\mathcal{E}$ smaller if necessary we can assume that $\hat{x}_i\in\Upsilon_i(B_{R/4})$ for $i=1,\ldots,n$. Denote $y_i=\Upsilon_i^{-1}(\hat{x}_i)$ for $i=1,\ldots,n$ and define $B_i=(\mathrm d\Upsilon_i|_{y_i})^{-1}\circ\hat{A}_i$ for $i=1,\ldots,n$. Since $\Upsilon_i^*(\omega)=\omega'$, $B_i\in Sp(2m,\mathbb{R})$ and so $(B_i,y_i)\in Sp(2m,\mathbb{R})\ltimes\mathbb{R}^{2m}$. Here $Sp(2m)$ is the automorphism group of $(\mathbb{R}^{2m},\omega')$. Using standard techniques from symplectic geometry we can now define families $\{\Psi^e_i\}_{e\in\mathcal{E}}$ of diffeomorphisms $\Psi^e_i:B_R\rightarrow B_R$ for $i=1,\ldots,n$, which depend smoothly on $e\in\mathcal{E}$, such that \rm{(a)}, \rm{(b)}, and \rm{(c)} hold, and such that $\Psi^e_i=(B_i,y_i)$ on $B_{R/4}$ for $i=1,\ldots,n$. But then by definition of $(B_i,y_i)$ we see that {\rm(d)} holds for $i=1,\ldots,n$.

Now we define $\Psi^e_M:M\rightarrow M$ to be $\Upsilon_i\circ\Psi^e_i\circ\Upsilon_i^{-1}$ on $\Upsilon_i(B_R)$ for $i=1,\ldots,n$ and the identity on $M\backslash\bigcup_{i=1}^n\Upsilon_i(B_R)$. This is clearly possible, since $\Upsilon_1(B_R),\ldots,\Upsilon_n(B_R)$ are pairwise disjoint in $M$. Since $\Psi^e_i$ satisfies $\rm{(a)}-\rm{(d)}$ for $i=1,\ldots,n$, it follows that $\Psi^e_M:M\rightarrow M$ is a family of smooth diffeomorphisms of $M$, which depends smoothly on $e\in\mathcal{E}$ and satisfies $\rm{(i)}-\rm{(iv)}$.
\end{proof}

Using Lemma \ref{LEMUU} we can now state a Lagrangian neighbourhood theorem which gives us a family of Lagrangian neighbourhoods for Lagrangian submanifolds with isolated conical singularities that are close to $F:L\rightarrow M$.

\begin{Thm}\label{LagrangianNeighbourhoodConeFamily1}
Let $(M,J,\omega,\Omega)$ be an $m$-dimensional almost Calabi--Yau manifold and $F:L\rightarrow M$ a Lagrangian submanifold with isolated conical singularities $x_1,\ldots,x_n$ and cones $C_1,\ldots,C_n$ as in Definition \ref{DefLagrangianConicalSingularities}. Moreover let $\Phi_{C_i}:U_{C_i}\rightarrow\mathbb{C}^m$ be a Lagrangian neighbourhood for $\iota_i:C_i\rightarrow\mathbb{C}^m$ for $i=1,\ldots,n$ as given by Theorem \ref{LagrangianNeighbourhoodCone} and let $\Phi_L:U_L\rightarrow M$ be a Lagrangian neighbourhood for $F:L\rightarrow M$ as given by Theorem \ref{LagrangianNeighbourhoodConicalSingularities}. Finally let $\mathcal{E}$ and $e_0$ be as above and choose $\{\Psi^e_M\}_{e\in\mathcal{E}}$ as given by Lemma \ref{LEMUU}.

Define smooth families $\{\Upsilon_i^e\}_{e\in\mathcal{E}}$ of embeddings $\Upsilon^e_i:B_R\rightarrow M$ by $\Upsilon^{e}_i=\Psi^{e}_M\circ\Upsilon_i$ for $i=1,\ldots,n$, and a smooth family $\{\Phi^e_L\}_{e\in\mathcal{E}}$ of embeddings $\Phi^e_L:U_L\rightarrow M$ by $\Phi^e_L=\Psi^e_M\circ\Phi_L$. Then $\{\Upsilon_i^e\}_{e\in\mathcal{E}}$ and $\{\Phi^e_L\}_{e\in\mathcal{E}}$ depend smoothly on $e\in\mathcal{E}$ for $i=1,\ldots,n$, and\vspace{0.17cm}
\begin{compactenum}
\item[{\rm(i)}] $\Upsilon_i^{e_0}=\Upsilon_i$ and $(\Upsilon_i^e)^*(\omega)=\omega'$, and for every $e\in\mathcal{E}$ with $e=(\hat{x}_1,\hat{A}_1,\ldots,\hat{x}_n,\hat{A}_n)$ we have that $\Upsilon^e_i(0)=\hat{x}_i$, and $\mathrm d\Upsilon_i^e(0)=\hat{A}_i$,
\item[{\rm(ii)}] $\Phi^{e_0}_L=\Phi_L$ and $(\Phi^e_L)^*(\omega)=\hat{\omega}$, and for every $e\in\mathcal{E}$ we have that $\Phi^e_L\equiv\Phi_L$ on $\pi^{-1}(L\backslash\bigcup_{i=1}^nS_i')\subset U_L$, where $S_i'=\phi_i(\Sigma_i\times(0,\frac{R}{2}))$.\vspace{0,2cm}
\end{compactenum}

\noindent Moreover $\Phi^e_L\circ\mathrm d\phi_i=\Upsilon_i^e\circ\Phi_{C_i}\circ\mathrm da_i$ on $T^*(\Sigma_i\times(0,\frac{R}{2}))$ for every $e\in\mathcal{E}$ and $i=1,\ldots,n$. Finally for $e\in\mathcal{E}$ with $e=(\hat{x}_1,\hat{A}_1,\ldots,\hat{x}_n,\hat{A}_n)$, $\Phi^e_L\circ 0:L\rightarrow M$ is a Lagrangian submanifold with isolated conical singularities $\hat{x}_1,\ldots,\hat{x}_n\in M$ modelled on the special Lagrangian cones $C_1,\ldots,C_n$ with isomorphisms $\hat{A}_i:T_{\hat{x}_i}M\rightarrow\mathbb{C}^m$ for $i=1,\ldots,n$ as in Definition \ref{DefLagrangianConicalSingularities}.
\end{Thm}

We need a slightly extended version of Theorem \ref{LagrangianNeighbourhoodConeFamily1}. Recall that every element in $\mathfrak{u}(m)\oplus\mathbb{C}^m\oplus\mathbb{R}$ gives rise to a unique moment map as explained in \S\ref{SLCONESQQQ}. What we have done until now is to use the $U(m)\ltimes\mathbb{C}^m$-moment maps to construct a family of Lagrangian neighbourhoods for the Lagrangian submanifolds with conical singularities that are close to $F:L\rightarrow M$. So what we have not taken into account so far is the $\mathbb{R}$-part of the moment maps. In order to do this let us choose functions $q_1,\ldots,q_n\in C^{\infty}(L)$ with $q_i\equiv 1$ on $S_i'$ and $q_i\equiv 0$ on $M\backslash S_i$ for $i=1,\ldots,n$, where $S_i'=\phi_i(\Sigma_i\times(0,\frac{R}{2}))$ for $i=1,\ldots,n$. Moreover let $\mathcal{U}_1,\ldots,\mathcal{U}_n\subset\mathbb{R}$ be small open neighbourhoods of the origin in $\mathbb{R}$, define $\mathcal{F}_i=\mathcal{E}_i\times\mathcal{U}_i$ for $i=1,\ldots,n$, and denote $\mathcal{F}=\mathcal{F}_1\times\cdots\times\mathcal{F}_n$. Then we define an open neighbourhood $U_L'$ of the zero section in $T^*L$ by
\[
U_L'=\left\{(x,\beta)\in U_L\;:\;\Gamma_{\beta+\sum_{i=1}^nc_i\mathrm dq_i}\subset U_L\mbox{ for every }c_i\in\mathcal{U}_i\right\}
\]
and we define a family $\{\Phi^f_L\}_{f\in\mathcal{F}}$ of Lagrangian neighbourhoods $\Phi^f_L:U_L\rightarrow M$ by $\Phi^f_L=\Phi^e_L\circ\sum_{i=1}^nc_i\mathrm dq_i$, where $f=(e_1,c_1,\ldots,e_n,c_n)$ and $e=(e_1,\ldots,e_n)\in\mathcal{E}$. We denote $f_0=(e_1,0,\ldots,e_n,0)$, where $e_0=(e_1,\ldots,e_n)$. Then $\Phi^{f_0}_L=\Phi_L$ is a Lagrangian neighbourhood for $F:L\rightarrow M$ and $\{\Phi^f_L\}_{f\in\mathcal{F}}$ is a family of Lagrangian neighbourhoods for the Lagrangian submanifolds with isolated conical singularities that are close to $F:L\rightarrow M$, which depends smoothly on $f\in\mathcal{F}$.

\section{Mean curvature flow of Lagrangian submanifolds with conical singularities}\label{DITISKLASSE}
\subsection{Setting up the short time existence problem}
From now on we fix an almost Calabi--Yau manifold $(M,J,\omega,\Omega)$, we define $\psi\in C^{\infty}(M')$ as in \eq{DefinitionPsi}, and we fix a Lagrangian submanifold $F_0:L\rightarrow M$ with isolated conical singularities $x_1,\ldots,x_n$ and model cones $C_1,\ldots,C_n$ as in Definition \ref{DefLagrangianConicalSingularities}. Later on we will also assume that the special Lagrangian cones $C_1,\ldots,C_n$ are stable. Moreover we define the manifold $\mathcal{F}$ as in \S\ref{Neighbourhoods} and we choose a family of Lagrangian neighbourhoods $\{\Phi^f_L\}_{f\in\mathcal{F}}$ as given by Theorem \ref{LagrangianNeighbourhoodConeFamily1} and the discussion following that theorem. Then $\Phi^{f_0}_L:U_L'\rightarrow M$ is a Lagrangian neighbourhood for $F_0:L\rightarrow M$.

The main difference between the setup of the short time existence problem for the generalized Lagrangian mean curvature flow of $F_0:L\rightarrow M$ to the one we chose in \S\ref{GULLI} for a compact Lagrangian submanifold is that we have to take into account that the parameter $f\in\mathcal{F}$ will change during the flow. In particular, when we try to find an integrated form for the generalized Lagrangian mean curvature flow, we expect to find an equation that not only involves the potential function $u$, which is a function on $L$ with a reasonable decay rate near each singularities, as in \S\ref{GULLI} but that also involves the parameter $f\in\mathcal{F}$, which describes the motion of the conical singularities.

In order to find the integrated form for the generalized Lagrangian mean curvature flow of $F_0:L\rightarrow M$, we first need to study the deformation vector field of the family $\{\Phi^f_L\}_{f\in\mathcal{F}}$. Let $\boldsymbol\mu\in\mathbb{R}^n$ with $\boldsymbol\mu>2$ and $u\in C^{\infty}_{\boldsymbol\mu}(L)$ with $\Gamma_{\mathrm du}\subset U_L'$. Then for every $f\in\mathcal{F}$, $\Phi^f_L\circ\mathrm du:L\rightarrow M$ is a Lagrangian submanifold with conical singularities modelled on $C_1,\ldots,C_n$. Let $v\in T\mathcal{F}$. If we differentiate $\Phi^f_L\circ\mathrm du$ with respect to $f$ in direction of $v$, then we obtain a section $\partial_v(\Phi^f_L\circ\mathrm du)$ of the vector bundle $(\Phi^f_L\circ\mathrm du)^*(TM)$ over $L$. Now recall how $\Phi^f_L:U_L'\rightarrow M$ was constructed. We started with a fixed Lagrangian neighbourhood $\Phi^{f_0}_L:U_L'\rightarrow M$ for $F_0:L\rightarrow M$ and then we obtained Lagrangian neighbourhoods for the nearby Lagrangian submanifolds with isolated conical singularities by applying $U(m)$-rotations and $\mathbb{C}^m$-translations near the conical singularities. When we studied deformations of special Lagrangian cones in $\mathbb{C}^m$ in \S\ref{SLCONESQQQ}, we introduced the idea of moment maps, which are the Hamiltonian potentials of the $\mathfrak{u}(m)\oplus\mathbb{C}^m\oplus\mathbb{R}$-vector fields on $\mathbb{C}^m$. Now in our situation, $\partial_v(\Phi^f_L\circ\mathrm du)$ corresponds, in an asymptotic sense at least, to one of these vector fields, and therefore $\partial_v(\Phi^f_L\circ\mathrm du)$ should have a Hamiltonian potential that is asymptotic to a moment map. In fact we have the following important proposition which can be found in \cite[\S 9.1]{Behrndt1}.

\begin{Prop}\label{MapXi}
Let $\boldsymbol\mu\in\mathbb{R}^n$ with $\boldsymbol\mu>2$, $f\in\mathcal{F}$, and $v\in f^*(T\mathcal F)$. Then there exists a smooth vector field $X_f(v)$ on $M$ that depends linearly on $v$ and satisfies
\[
(\Phi^f_L\circ\mathrm du)^*(X_f(v))=\partial_v(\Phi^f_L\circ\mathrm du)
\]
for every $u\in C^{\infty}_{\boldsymbol\mu}(L)$ with $\Gamma_{\mathrm du}\subset U_L'$. Moreover there exists $H_f(v)\in C^{\infty}(M)$ that depends linearly on $v\in f^*(T\mathcal F)$ such that $d[H_f(v)]=X_f(v)\;\lrcorner\;\omega$. 

Finally, if $\boldsymbol\mu\in\mathbb{R}^n$ with $2<\boldsymbol\mu<\boldsymbol\nu$ and $u\in C^{\infty}_{\boldsymbol\mu}(L)$ with $\Gamma_{\mathrm du}\subset U_L'$ we define
\[
\Xi_{(u,f)}:f^*(T\mathcal F)\longrightarrow C^{\infty}(M),\quad\Xi_{(u,f)}(v)=(\Phi^f_L\circ\mathrm du)^*(H_f(v)).
\]
Then
\[
\mathrm d[\Xi_{(u,f)}(v)]=(\Phi^f_L\circ\mathrm du)^*(\partial_v(\Phi^f_L\circ\mathrm du)\;\lrcorner\;\omega),
\]
and $\Xi_{(u,f)}(v)$ is asymptotic to a moment map in the following sense.

Let $f=(\hat{e}_1,c_1,\ldots,\hat{e}_n,c_n)\in\mathcal{F}$ with $\hat{e}_i=(\hat{x}_i,\hat{A}_i)$ for $i=1,\ldots,n$, denote $f_i=(\hat{x}_i,\hat{A}_i,c_i)$ for $i=1,\ldots,n$, and let $v=(v_1,\ldots,v_n)\in T_{f_1}\mathcal{F}_1\oplus\cdots\oplus T_{f_n}\mathcal{F}_n$. Then the following holds.\vspace{0.17cm}
\begin{compactenum}
\item[{\rm(i)}] If $v_i\in T_{\hat{A}_i}\mathcal{A}_{\hat{x}_i}$, then there exists a unique $X_i\in\mathfrak{u}(m)\oplus\mathbb{R}$, such that 
\[
\bigl|\nabla^j(\phi_i^*(\Xi_{(u,f)}(v))-\iota_i^*(\mu_{X_i}))\bigr|=O(r^{\mu_i-j})\quad\mbox{as }r\rightarrow 0
\]
for $j\in\mathbb{N}$.
\item[{\rm(ii)}] If $v_i\in T_{\hat{x}_i}M$, then there exists a unique $X_i\in\mathbb{C}^m\oplus\mathbb{R}$, such that 
\[
\bigl|\nabla^j(\phi_i^*(\Xi_{(u,f)}(v))-\iota_i^*(\mu_{X_i}))\bigr|=O(r^{\mu_i-1-j})\quad\mbox{as }r\rightarrow 0
\]
for $j\in\mathbb{N}$.\vspace{0.17cm}
\end{compactenum}
\end{Prop}

Using Proposition \ref{MapXi} we are now ready to integrate the generalized Lagrangian mean curvature flow for $F_0:L\rightarrow M$. By Proposition \ref{PPP} we can choose a closed one-form $\beta_0$ that represents the Maslov class of $F_0:L\rightarrow M$ and that is supported on $K$. Then, in particular, $\beta_0$ is zero near each conical singularity. Let $T>0$ be small and define a one-parameter family $\{\beta(t)\}_{t\in(0,T)}$ of closed one-forms by $\beta(t)=t\beta_0$ for $t\in(0,T)$. Then $\{\beta(t)\}_{t\in(0,T)}$ extends continuously to $t=0$ with $\beta(0)=0$. Finally we choose $\Theta$ as in \S\ref{GULLI} and then we define an operator $P$ as follows. The domain of the operator $P$ is given by
\begin{align*}
&\mathcal{D}=\left\{(u,f)\;:\;u\in C^{\infty}((0,T)\times L),f\in C^{\infty}((0,T);\mathcal{F}),\;\right.\\&\left.\quad\quad\quad\;u\mbox{ and }f\mbox{ extend continuously to }t=0,\;\Gamma_{\mathrm du(t,\cdot)+\beta(t)}\subset U_L'\mbox{ for }t\in(0,T)\right\}
\end{align*}
and we define $P:\mathcal{D}\longrightarrow C^{\infty}((0,T)\times L)$ by 
\begin{equation*}
P(u,f)=\frac{\partial u}{\partial t}-\Theta(\Phi^f_{L}\circ(\mathrm du+\beta))-\Xi_{(u,f)}\left(\frac{\mathrm df}{\mathrm dt}\right).
\end{equation*}

We now consider the following Cauchy problem
\begin{align}\label{CauchyProblem2}
\begin{split}
&P(u,f)(t,x)=0\quad\mbox{for }(t,x)\in(0,T)\times L,\\ &u(0,x)=0\quad\quad\quad\;\;\mbox{for }x\in L,\\&f(0)=f_0
\end{split}
\end{align}
and we show that this is in fact an integrated version of the generalized Lagrangian mean curvature flow of $F_0:L\rightarrow M$.

\begin{Prop}\label{HHUU}
Let $(u,f)\in\mathcal{D}$ be a solution of the Cauchy problem \eq{CauchyProblem2} and define
\[
F(t,\cdot):L\rightarrow M,\quad F(t,\cdot)=\Phi^{f(t)}_L\circ\mathrm du(t,\cdot).
\]
Then $\{F(t,\cdot)\}_{t\in(0,T)}$ is a one-parameter family of Lagrangian submanifolds with isolated conical singularities that evolves by generalized Lagrangian mean curvature flow with initial condition $F_0:L\rightarrow M$.
\end{Prop}

\begin{proof}
It is clear that $F(0,x)=F_0(x)$ for $x\in L$. In order to show that $\{F(t,\cdot)\}_{t\in(0,T)}$ evolves by generalized Lagrangian mean curvature flow it suffices to show that $\alpha_{\frac{\partial F}{\partial t}}=\alpha_K$. Let us assume for simplicity that $\psi\equiv 0$, i.e. $M$ is Calabi--Yau, and that $F_0:L\rightarrow M$ has zero Maslov class. Then $\beta\equiv 0$ and $\Theta=\theta$. Using Lemma \ref{Var} we find that
\[
\frac{\mathrm d}{\mathrm dt}F(t,\cdot)=\partial_{\frac{\mathrm df}{\mathrm dt}}(\Phi^f_L\circ\mathrm du)-\alpha^{-1}(\mathrm d[\partial_tu])+V(\mathrm d[\partial_tu]).
\]
Using that $P(u,f)=0$, the Lagrangian property, and the definition of $\alpha$ from \eq{Defalpha} it follows that
\[
\alpha_{\frac{\partial F}{\partial t}}=F^*(\partial_{\frac{\mathrm df}{\mathrm dt}}(\Phi^f_L\circ\mathrm du)\;\lrcorner\;\omega)-\mathrm d[\theta(F)]-\Xi_{(u,f)}\left(\frac{\mathrm df}{\mathrm dt}\right).
\]
Now recalling the definition of $\Xi_{(u,f)}(\frac{\mathrm df}{\mathrm dt})$ from Proposition \ref{MapXi} we conclude that $\alpha_{\frac{\partial F}{\partial t}}=-\mathrm d[\theta(F)]$. But $\alpha_K=-\mathrm d[\theta(F)]$ by Proposition \ref{GeneralizedMeanCurvatureForm}, and therefore $\alpha_{\frac{\partial F}{\partial t}}=\alpha_K$ as we wanted to show.
\end{proof}

As in the short time existence problem for the generalized Lagrangian mean curvature flow when the initial Lagrangian submanifold is compact, we are now left with studying existence and regularity of solutions to the Cauchy problem \eq{CauchyProblem2}. In the next sections we will discuss how short time existence of solutions with low regularity to the Cauchy problem \eq{CauchyProblem2} is proved and we will also discuss the regularity of these solutions in spatial and time direction.

\subsection{Smoothness of the operator $P$ as a map between Banach manifolds}
In this section we show how the operator $P:\mathcal{D}\rightarrow C^{\infty}((0,T)\times L)$ can be extended to a smooth operator between certain Banach manifolds. Once this has been done we can use linearization techniques to prove short time existence of solution with low regularity to the Cauchy problem \eq{CauchyProblem2}. From now on we will assume for simplicity that $M$ is Calabi--Yau and that $F_0:L\rightarrow M$ has zero Maslov class. Then $\psi\equiv 0$, $\beta\equiv 0$, and $\Theta\equiv\theta$.

We first extend the domain of the operator $P$. The manifold $\mathcal{F}$ embeds into $\mathbb{R}^{s}$ for some sufficiently large $s\in\mathbb{N}$. Let $p\in(1,\infty)$ and $f\in W^{1,p}((0,T);\mathbb{R}^s)$. Then $f:(0,T)\rightarrow\mathcal{F}$ is continuous by the Sobolev Embedding Theorem and the condition $f(t)\in\mathcal{F}$ makes sense for every $t\in(0,T)$. We define the Banach manifold $W^{1,p}((0,T);\mathcal{F})$ by
\[
W^{1,p}((0,T);\mathcal{F})=\left\{f\in W^{1,p}((0,T);\mathbb{R}^s)\;:\;f(t)\in\mathcal{F}\mbox{ for }t\in(0,T)\right\}.
\]
Now let $k\in\mathbb{N}$, $p\in(1,\infty)$ with $k-\frac{m}{p}>2$, and $\boldsymbol\mu\in\mathbb{R}^n$ with $2<\boldsymbol\mu<\boldsymbol\nu$. For $T>0$ we define
\begin{gather*}
\mathcal{D}_{\boldsymbol\mu}^{k,p}=\bigl\{(u,f)\;:\;u\in W^{1,k,p}_{\boldsymbol\mu}((0,T)\times L),\;f\in W^{1,p}((0,T);\mathcal{F}),\;\;\;\;\;\;\;\;\;\;\;\;\;\;\\\;\;\;\;\;\;\;\;\;\;\;\;\;\;\;\;\;\;\;\;\;\;\;\;\;\;\;\;\;\;\;\;\;\;\;\;\;\;\;\;\;\;\;\;\;\;\mbox{such that } \Gamma_{\mathrm du(t,.)}\subset U_L'\mbox{ for }t\in(0,T)\bigr\}.
\end{gather*}
Let $(u,f)\in\mathcal{D}_{\boldsymbol\mu}^{k,p}$. Since $k-\frac{m}{p}>2$, it follows from the Sobolev Embedding Theorem that $\Phi^{f(t)}_L\circ\mathrm du(t,\cdot):L\rightarrow M$ is a Lagrangian submanifold for almost every $t\in(0,T)$. In particular $\theta(\Phi^{f(t)}_L\circ\mathrm du(t,\cdot))$ is well defined for almost every $t\in(0,T)$ and therefore $P$ acts on $\mathcal{D}_{\boldsymbol\mu}^{k,p}$. 

In order to define the target space for $P$ acting on $\mathcal{D}^{k,p}_{\boldsymbol\mu}$ we define for $k\in\mathbb{N}$, $p\in(1,\infty)$, and $\boldsymbol\gamma\in\mathbb{R}^n$ with $\boldsymbol\gamma>0$ a weighted parabolic Sobolev space $W^{k,p}_{\boldsymbol\gamma,\mathsf{Q}}(L)$ with discrete asymptotics by
\[
W^{k,p}_{\boldsymbol\gamma,\mathsf{Q}}(L)=W^{k,p}_{\boldsymbol\gamma}\oplus\Span\{q_1,\ldots,q_n\},
\]
where the functions $q_1,\ldots,q_n$ are defined as in the end of \S\ref{Neighbourhoods}. Further we define the weighted parabolic Sobolev space $W^{0,k-2,p}_{\boldsymbol\mu-2,\mathsf{Q}}((0,T)\times L)$ with discrete asymptotics by
\begin{gather*}
W^{0,k-2,p}_{\boldsymbol\mu-2,\mathsf{Q}}((0,T)\times L)=L^p((0,T);W^{k-2,p}_{\boldsymbol\mu-2,\mathsf{Q}}(L))
\end{gather*}
and we then have the following result.

\begin{Prop}\label{Smoothness}
Let $\boldsymbol\mu\in\mathbb{R}^n$ with $2<\boldsymbol\mu<\boldsymbol\nu$. Then, for $k\in\mathbb{N}$ and $p\in(1,\infty)$ sufficiently large, the operator
\begin{equation}\label{OP}
P:\mathcal{D}^{k,p}_{\boldsymbol\mu}\longrightarrow W^{0,k-2,p}_{\boldsymbol\mu-2,\mathsf{Q}}((0,T)\times L)
\end{equation}
is a smooth operator between Banach manifolds.
\end{Prop}

The proof of Proposition \ref{Smoothness} is more or less straightforward but rather long and technically complicated. The main step in the proof is to show that the operator $P$ in \eq{OP} is well defined, and the difficult part is here to show that the $\theta$-term lies in the correct space. In fact, it is not hard to see that
\[
\frac{\partial u}{\partial t},\;\Xi_{(u,f)}\left(\frac{\mathrm df}{\mathrm dt}\right)\in W^{0,k-2,p}_{\boldsymbol\mu-2,\mathsf{Q}}((0,T)\times L)
\]
for $(u,f)\in\mathcal{D}^{k,p}_{\boldsymbol\mu}$ and sufficiently large $k\in\mathbb{N}$ and $p\in(1,\infty)$. Thus it only remains to show that 
\begin{equation}\label{HHH}
\theta(\Phi^f_{L}\circ\mathrm du)\in W^{0,k-2,p}_{\boldsymbol\mu-2,\mathsf{Q}}((0,T)\times L).
\end{equation}
Let us define a function $F$ that maps
\begin{align*}
F:\bigl\{(f,x,y,z)\;:\;f\in\mathcal{F},\;x\in L,\;y\in T_x^*L\cap U_L',\;z\in\otimes^2T_x^*L\bigr\}\longrightarrow\mathbb{R}.
\end{align*}
by $F(f,x,\mathrm du(x),\nabla \mathrm du(x))=\theta(\Phi^f_L\circ\mathrm du)(x)$. Then $F$ is a smooth and nonlinear function on its domain, since $\Omega,g$, and $\Phi^f_L$ are smooth and $\Phi^f_L$ depends smoothly on $f\in\mathcal{F}$. Furthermore we define a function $Q$ on the domain of $F$ by
\begin{equation*}
\begin{split}
&Q(f,x,y,z)=F(f,x,y,z)-F(f_0,x,0,0)\\&\;\;\;\;\;\;\;\;\;\;\;\;\;\;\;\;\;\;\;\;\;\;\;\;\;-(\partial_yF)(f_0,x,0,0)\cdot y-(\partial_zF)(f_0,x,0,0)\cdot z.
\end{split}
\end{equation*}
Since $F$ is smooth, $Q$ is also a smooth and nonlinear function on its domain.

The main step in the proof of \eq{HHH} is to show the following lemma, which can be found either in Joyce \cite[Prop. 6.3]{Joyce2} or in the author's thesis \cite[Lem. 9.8]{Behrndt1}.

\begin{Lem}\label{HELP}
For $(u,f)\in\mathcal{D}^{k,p}_{\boldsymbol\mu}$ we have that
\begin{align*}
&F(f,\cdot,\mathrm du,\nabla\mathrm du)=F(f_0,\cdot,0,0)+\Delta u-\mathrm d[\theta(F_0)](\hat{V}(\mathrm du))+Q(f,\cdot,\mathrm du,\nabla\mathrm du).
\end{align*}
Moreover, for $a,b,c\geq 0$ and small $\rho^{-1}(x)|y|$, $|z|$, and $d(f,f_0)$ the function $Q$ satisfies
\begin{gather*}
(\nabla_x)^a(\partial_y)^b(\partial_z)^cQ(f,x,y,z)=\quad\quad\quad\quad\quad\quad\quad\quad\quad\quad\quad\quad\quad\quad\quad\quad\quad\quad\quad\quad\\\quad O\left(\rho(x)^{-a-\max\{2,b\}}|y|^{\max\{0,2-b\}}+\rho(x)^{-a}|z|^{\max\{0,2-c\}}+\rho(x)^{1-a-b}d(f,f_0)\right),
\end{gather*}
uniformly for $x\in L$. Here $d(f,f_0)$ denotes the distance of $f$ to $f_0$ in $\mathcal{F}$.

Moreover the function $F(f_0,\cdot,0,0)$ on $L$ admits an expansion of the form
\begin{equation}\label{POIE}
\phi_i^*(F(f_0,\cdot,0,0))=\theta_i+\Delta a_i+R_i(\cdot,\mathrm da_i,\nabla\mathrm da_i)
\end{equation}
for $i=1,\ldots,n$. Here $\theta_i$ is the Lagrangian angle of the special Lagrangian cone $C_i$ and, for $a,b,c\geq 0$ and small $r^{-1}|y|$ and $|z|$, the error term $R_i$ satisfies
\begin{gather*}
(\nabla_x)^a(\partial_y)^b(\partial_z)^cR_i(\sigma,r,y,z)=\quad\quad\quad\quad\quad\quad\quad\quad\quad\quad\quad\quad\quad\quad\quad\quad\quad\quad\quad\quad\\\quad O\left(r^{-a-\max\{2,b\}}|y|^{\max\{0,2-b\}}+r^{-a}|z|^{\max\{0,2-c\}}+r^{1-a-b}\right),
\end{gather*}
uniformly for $x=(\sigma,r)\in\Sigma_i\times(0,R)$.
\end{Lem}

Now, in order to show that \eq{HHH} holds, we can expand the function $F$ as in Lemma \ref{HELP} and show that each of the terms in the expansion lies in the space $W^{0,k-2,p}_{\boldsymbol\mu-2,\mathsf{Q}}((0,T)\times L)$. It is clear that $\Delta u$ lies in $W^{0,k-2,p}_{\boldsymbol\mu-2,\mathsf{Q}}((0,T)\times L)$. Moreover, once $k\in\mathbb{N}$ and $p\in(1,\infty)$ are chosen sufficiently large one can use the estimates for the function $Q$ from the lemma, interpolation results for weighted parabolic Sobolev spaces, and the Sobolev Embedding Theorem to show that in fact $Q(f,\cdot,\mathrm du,\nabla\mathrm du)\in W^{0,k-2,p}_{\boldsymbol\mu-2,\mathsf{Q}}((0,T)\times L)$. Thus we are only left with showing that $F(f_0,\cdot,0,0)$ and $\mathrm d[\theta(F_0)](\hat{V}(\mathrm du))$ both lie in $W^{0,k-2,p}_{\boldsymbol\mu-2,\mathsf{Q}}((0,T)\times L)$. Using the expansion of $F(f_0,\cdot,0,0)$ in \eq{POIE}, the fact that $a\in C^{\infty}_{\boldsymbol\nu}(L)$, and the estimates for $R_i$ one can show with the same methods as before that in fact $F(f_0,\cdot,0,0)\in W^{0,k-2,p}_{\boldsymbol\mu-2,\mathsf{Q}}((0,T)\times L)$. In a similar way it then follows that $\mathrm d[\theta(F_0)](\hat{V}(\mathrm du))$ lies in the space $W^{0,k-2,p}_{\boldsymbol\mu-2,\mathsf{Q}}((0,T)\times L)$ and therefore that \eq{HHH} holds. This shows that $P$ in \eq{OP} is well defined and then, using the smoothness of the function $Q$ and the Mean Value Theorem \cite[XIII, \S 4]{Lang}, it is straightforward to show that the operator $P$ in \eq{OP} is in fact smooth. The detailed proof of Proposition \ref{Smoothness} can be found in the author's thesis \cite[\S 9.3]{Behrndt1}.

\subsection{Short time existence of the flow with low regularity}
Before we show how to prove short time existence of solutions with low regularity to the Cauchy problem \eq{CauchyProblem2}, we need to discuss the operator $P$ and its linearization in more detail.

Let $(u,f)\in\mathcal{D}_{\boldsymbol\mu}^{k,p}$. Then $t\mapsto\Xi_{(u(t,\cdot),f(t))}$ is a section of the vector bundle $f^*(\Hom(T\mathcal{F},C^{1}_{\loc}(L)))$ over the manifold $(0,T)$. Define
\begin{equation*}
V_{\mathsf{P}_{(u,f)}}(L)=\im\;\bigl\{\Xi_{(u,f)}:f^*(T\mathcal{F})\longrightarrow C^1_{\loc}(L)\bigr\}.
\end{equation*}
Then $V_{\mathsf{P}_{(u,f)}}(L)$ is a finite dimensional vector bundle over $(0,T)$ with fibres of dimension $\dim\mathcal{F}$. Also note that if $u$ is smooth, then each fibre of $V_{\mathsf{P}_{(u,f)}}(L)$ consists of smooth functions on $L$. Since $\Xi_{(u,f)}(v)$ is asymptotic to a moment map for each $v\in f^*(T\mathcal{F})$ by Proposition \ref{MapXi}, it follows that for every $(u,f)\in\mathcal{D}_{\boldsymbol\mu}^{k,p}$, $V_{\mathsf{P}_{(u,f)}}(L)$ has trivial intersection with $L^p_{\boldsymbol\mu}(L)$ in each fibre over $(0,T)$. Hence we can define
\[
W^{k,p}_{\boldsymbol\mu,\mathsf{P}_{(u,f)}}(L)=W^{k,p}_{\boldsymbol\mu}(L)\oplus V_{\mathsf{P}_{(u,f)}}(L).
\]
Then $W^{k,p}_{\boldsymbol\mu,\mathsf{P}_{(u,f)}}(L)$ is a Banach bundle over the Banach manifold $\mathcal{D}_{\boldsymbol\mu}^{k,p}$ with fibres being weighted Sobolev spaces with discrete asymptotics. If $u$ and $f$ are constant in time, so for instance at the initial condition $u=0$ and $f=f_0$, then $W^{k,p}_{\boldsymbol\mu,\mathsf{P}_{(u,f)}}(L)$ is simply a weighted Sobolev space with discrete asymptotics, and the discrete asymptotics are defined using the map $\Xi_{(u,f)}:f^*(T\mathcal{F})\rightarrow C^1_{\loc}(L)$.

Next we define the weighted parabolic Sobolev space $W^{1,k,p}_{\boldsymbol\mu,\mathsf{P}_{(u,f)}}((0,T)\times L)$ with discrete asymptotics in the usual way by
\[
W^{1,k,p}_{\boldsymbol\mu,\mathsf{P}_{(u,f)}}((0,T)\times L)=L^p((0,T);W^{k,p}_{\boldsymbol\mu,\mathsf{P}_{(u,f)}}(L))\cap W^{1,p}((0,T);W^{k-2,p}_{\boldsymbol\mu-2,\mathsf{Q}}(L)).
\]
Now consider the linearization of the operator \eq{OP} at some $(u,f)\in\mathcal{D}^{k,p}_{\boldsymbol\mu}$, which is a linear operator
\[
\mathrm dP(u,f):W^{1,k,p}_{\boldsymbol\mu}((0,T)\times L)\oplus W^{1,p}((0,T);f^*(T\mathcal{F}))\longrightarrow W^{0,k-2,p}_{\boldsymbol\mu-2,\mathsf{Q}}((0,T)\times L).
\]
Using $\Xi_{(u,f)}$ to identify $f^*(T\mathcal{F})$ with $V_{\mathsf{P}_{(u,f)}}(L)$, we can understand the linearization of \eq{OP} at $(u,f)$ as a linear operator
\begin{equation*}
\mathrm dP(u,f):W^{1,k,p}_{\boldsymbol\mu,\mathsf{P}_{(u,f)}}((0,T)\times L)\longrightarrow W^{0,k-2,p}_{\boldsymbol\mu-2,\mathsf{Q}}((0,T)\times L).
\end{equation*}

In the next proposition we obtain an explicit formula for the linearization of \eq{OP} at the initial condition $(0,f_0)$, but first we need to introduce some more notation. Let $(u,f)\in\mathcal{D}^{k,p}_{\boldsymbol\mu}$ and $w\in f^*(T\mathcal{F})$. Then $\partial_w(\Phi^f_L\circ\mathrm du)$ is a section of the vector bundle $(\Phi^f_L\circ\mathrm du)^*(TM)$. From Proposition \ref{MapXi} it follows that the normal part of $\partial_w(\Phi^f_L\circ\mathrm du)$ is equal to $-J(\mathrm d(\Phi^f_L\circ\mathrm du)(\nabla\Xi_{(u,f)}(w)))$. We then define $-\hat{W}(\mathrm d[\Xi_{(u,f)}(w)])\in TL$ to be the tangential part of $\partial_w(\Phi^f_L\circ\mathrm du)$, i.e. we have that
\[
\partial_w(\Phi^f_L\circ\mathrm du)=-J(\mathrm d(\Phi^f_L\circ\mathrm du)(\nabla\Xi_{(u,f)}(w)))-\mathrm d(\Phi^f_L\circ\mathrm du)(\hat{W}(\mathrm d[\Xi_{(u,f)}(w)])),
\] 
and then the following result holds.

\begin{Prop}\label{Linearization}
Let $(u,f)\in\mathcal{D}_{\boldsymbol\mu}^{k,p}$ and $v-\Xi_{(0,f_0)}(w)\in W^{1,k,p}_{\boldsymbol\mu,\mathsf{P}_{(0,f_0)}}((0,T)\times L)$, where $v\in W^{1,k,p}_{\boldsymbol\mu}((0,T)\times L)$ and $w\in W^{1,p}((0,T);f^*(T\mathcal{F}))$. Then
\begin{align*}
\begin{split}
&\mathrm dP(0,f_0)(v,\Xi_{(0,f_0)}(w))=\frac{\partial}{\partial t}(v-\Xi_{(0,f_0)}(w))-\Delta(v-\Xi_{(0,f_0)}(w))\\&\quad\quad\quad\quad\quad\quad\quad\quad\quad\quad\quad\quad\quad+\mathrm d[\theta(F_0)](\hat{V}(\mathrm dv)-\hat{W}(\mathrm d[\Xi_{(0,f_0)}(w)])).
\end{split}
\end{align*}
Here the Laplace operator and $\nabla$ are computed using the Riemannian metric $F_0^*(g)$ on $L$.
\end{Prop}

\noindent The proof of Proposition \ref{Linearization} consists of a rather long computation and can be found in \cite[Prop. 9.10]{Behrndt1}.

We define 
\[
\tilde{\mathcal{D}}_{\boldsymbol\mu}^{k,p}=\bigl\{(u,f)\in\mathcal{D}_{\boldsymbol\mu}^{k,p}\;:\;u(0,\cdot)=0\mbox{ on }L,\;f(0)=f_0\bigr\}.
\]
Recall that if $(u,f)\in\mathcal{D}_{\boldsymbol\mu}^{k,p}$, then $u$ and $f$ extend continuously to $t=0$, since $u$ and $f$ are uniformly H\"older continuous on $(0,T)$ by the Sobolev Embedding Theorem. Moreover observe that $(u,f)\in\mathcal{D}_{\boldsymbol\mu}^{k,p}$ is a solution of the Cauchy problem \eq{CauchyProblem2} if and only if $(u,f)\in\tilde{\mathcal{D}}_{\boldsymbol\mu}^{k,p}$ and $P(u,f)=0$. We define
\[
\tilde{W}^{1,k,p}_{\boldsymbol\mu,\mathsf{P}_{(u,f)}}((0,T)\times L)=\bigl\{v\in W^{1,k,p}_{\boldsymbol\mu,\mathsf{P}_{(u,f)}}((0,T)\times L)\;:\;v(0,\cdot)=0\mbox{ on }L\bigr\}.
\]
In the next proposition we show that the linearization of the operator \eq{OP} at the initial condition $(0,f_0)$ is an isomorphism provided that the conical singularities are modelled on stable special Lagrangian cones.

\begin{Prop}\label{IsomorphismLinearization}
Assume that the model cones $C_1,\ldots,C_n$ of $F_0:L\rightarrow M$ are stable special Lagrangian cones in the sense of Definition \ref{STABILITY}. Then, for $T>0$ sufficiently small, the linear operator
\begin{equation}\label{uueeuu}
\mathrm dP(0,f_0):\tilde{W}^{1,k,p}_{\boldsymbol\mu,\mathsf{P}_{(0,f_0)}}((0,T)\times L)\longrightarrow W^{0,k-2,p}_{\boldsymbol\mu-2,\mathsf{Q}}((0,T)\times L)
\end{equation}
is an isomorphism of Banach spaces.
\end{Prop}

\begin{proof}
We only give a sketch of proof. Let us define an operator 
\begin{equation}\label{OPERH}
H:\tilde{W}^{1,k,p}_{\boldsymbol\mu,\mathsf{P}_{(0,f_0)}}((0,T)\times L)\longrightarrow W^{0,k-2,p}_{\boldsymbol\mu-2,\mathsf{Q}}((0,T)\times L)
\end{equation}
by
\begin{align*}
H(v,\Xi_{(0,f_0)}(w))&=\frac{\partial}{\partial t}(v-\Xi_{(0,f_0)}(w))-\Delta(v-\Xi_{(0,f_0)}(w))\quad\\&\quad\quad\quad\quad+\mathrm d[\theta(F_0)](\hat{W}(\mathrm d[v-\Xi_{(0,f_0)}(w)])).
\end{align*}
Observe that
\begin{equation}\label{OIAP}
\mathrm dP(0,f_0)(v,\Xi_{(0,f_0)}(w))-H(v,\Xi_{(0,f_0)}(w))=\mathrm d[\theta(F_0)](\hat{V}(\mathrm dv)-\hat{W}(\mathrm dv))
\end{equation}
and that each of the terms on the right side of \eq{OIAP} lies in $W^{0,k-2,p}_{\boldsymbol\mu-2,\mathsf{Q}}((0,T)\times L)$. Thus it follows that the operator $H$ in \eq{OPERH} is well defined. Let us also define $D:W^{k,p}_{\boldsymbol\mu,\mathsf{P}_{(0,f_0)}}(L)\rightarrow W^{k-2,p}_{\boldsymbol\mu-2,\mathsf{Q}}(L)$ by
\[
D(v,\Xi_{(0,f_0)}(w))=\Delta(v-\Xi_{(0,f_0)}(w))+\mathrm d[\theta(F_0)](\hat{W}(\mathrm d[v-\Xi_{(0,f_0)}(w)])).
\]
Then one can show that $D$ is an operator of Laplace type as defined in \S\ref{GGG} and we can define weighted Sobolev spaces with discrete asymptotics as in \S\ref{DiscreteASYMPTOTICS}. Notice that $H=\partial_t-D$. 

Now comes the key point about the stability of the special Lagrangian cones $C_1,\ldots,C_n$. Using the stability of $C_1,\ldots,C_n$ and Theorem \ref{Fredholm2} we find that 
\[
W^{k,p}_{\boldsymbol\mu,\mathsf{P}_{(0,f_0)}}(L)=W^{k,p}_{\boldsymbol\mu,\mathsf{P}^D_{\boldsymbol\mu}}(L)\quad\mbox{ and }\quad W^{k-2,p}_{\boldsymbol\mu-2,\mathsf{Q}}(L)=W^{k-2,p}_{\boldsymbol\mu-2,\mathsf{P}^D_{\boldsymbol\mu-2}}(L)
\]
and hence
\[
\tilde{W}^{1,k,p}_{\boldsymbol\mu,\mathsf{P}_{(0,f_0)}}((0,T)\times L)=\tilde{W}^{1,k,p}_{\boldsymbol\mu,\mathsf{P}^D_{\boldsymbol\mu}}((0,T)\times L)
\]
and also
\[
W^{0,k-2,p}_{\boldsymbol\mu-2,\mathsf{Q}}((0,T)\times L)=W^{0,k-2,p}_{\boldsymbol\mu-2,\mathsf{P}^D_{\boldsymbol\mu-2}}((0,T)\times L),
\]
where the weighted parabolic Sobolev spaces with discrete asymptotics are defined as in \S\ref{PIIEO}. In particular it follows that $H$ in \eq{OPERH} is a map
\begin{equation}\label{HJK}
H:\tilde{W}^{1,k,p}_{\boldsymbol\mu,\mathsf{P}^D_{\boldsymbol\mu}}((0,T)\times L)\longrightarrow W^{0,k-2,p}_{\boldsymbol\mu-2,\mathsf{P}^D_{\boldsymbol\mu-2}}((0,T)\times L)
\end{equation}
and Theorem \ref{SobolevRegularityHeat} and the Open Mapping Theorem \cite[XV, Thm. 1.3]{Lang} imply that \eq{HJK} is an isomorphism of Banach spaces.

Using interpolation estimates for weighted parabolic Sobolev spaces one can show that the operator $\mathrm dP(0,f_0)-H$ is a bounded operator
\[
\mathrm dP(0,f_0)-H:\tilde{W}^{1,k,p}_{\boldsymbol\mu,\mathsf{P}^D_{\boldsymbol\mu}}((0,T)\times L)\longrightarrow C^0((0,T);W^{k-2,p}_{\boldsymbol\mu-2}(L))
\]
and, using the Rellich--Kondrakov Theorem for weighted spaces and the Aubin--Dubinski\u{\i} Lemma as in \cite[Prop. 7.1]{Behrndt1}, one can further show that $\mathrm dP(0,f_0)-H$ is a compact operator $\tilde{W}^{1,k,p}_{\boldsymbol\mu,\mathsf{P}^D_{\boldsymbol\mu}}((0,T)\times L)\longrightarrow W^{0,k-2,p}_{\boldsymbol\mu-2,\mathsf{P}^D_{\boldsymbol\mu-2}}((0,T)\times L)$. Since \eq{OPERH} is an isomorphism, it is a Fredholm operator with index zero. Using a standard perturbation argument and the Fredholm alternative it then follows that for $T>0$ sufficiently small \eq{uueeuu} is an isomorphism.
\end{proof}

We are now ready to prove short time existence of solutions with low regularity to the Cauchy problem \eq{CauchyProblem2}.

\begin{Prop}\label{ExistenceSolutions}
Let $\boldsymbol\mu\in\mathbb{R}^n$ with $2<\boldsymbol\mu<\boldsymbol\nu$ and $(2,\mu_i]\cap\mathcal{E}_{\Sigma_i}=\emptyset$ for $i=1,\ldots,n$ and assume as in Proposition \ref{IsomorphismLinearization} that $C_1,\ldots,C_n$ are stable special Lagrangian cones. Then there exists $\tau>0$ and $(u,f)\in\tilde{\mathcal{D}}_{\boldsymbol\mu}^{k,p}$, such that $P(u,f)=0$ on the time interval $(0,\tau)$.
\end{Prop}

\begin{proof}
By Proposition \ref{IsomorphismLinearization},
\begin{equation*}
\mathrm dP(0,f_0):\tilde{W}^{1,k,p}_{\boldsymbol\mu,\mathsf{P}_{(0,f_0)}}((0,T)\times L)\longrightarrow W^{0,k-2,p}_{\boldsymbol\mu-2,\mathsf{Q}}((0,T)\times L)
\end{equation*}
is an isomorphism of Banach spaces. Since $P:\tilde{\mathcal{D}}_{\boldsymbol\mu}^{k,p}\rightarrow W^{0,k-2,p}_{\boldsymbol\mu-2,\mathsf{Q}}((0,T)\times M)$ is smooth by Proposition \ref{Smoothness}, the Inverse Function Theorem for Banach manifolds \cite[XIV, Thm. 1.2]{Lang} shows that there exist open neighbourhoods $V\subset\tilde{\mathcal{D}}_{\boldsymbol\mu}^{k,p}$ of $(0,f_0)$ and $W\subset W^{0,k-2,p}_{\boldsymbol\mu-2,\mathsf{Q}}((0,T)\times L)$ of $P(0,f_0)$, such that $P:V\rightarrow W$ is a smooth diffeomorphism. For $\tau\in(0,T)$ we define a function $w_{\tau}$ on $(0,T)\times L$ by
\[
w_{\tau}(t,x)=\left\{\begin{array}{ll}0 & \mbox{for }t<\tau\mbox{ and }x\in L,\\P(0,f_0)(t,x) & \mbox{for }t\geq\tau\mbox{ and }x\in L.\end{array}\right.
\]
Then $w_{\tau}\in W^{0,k-2,p}_{\boldsymbol\mu-2,\mathsf{Q}}((0,T)\times L)$ for every $\tau\in(0,T)$. In particular we can make $w_{\tau}-P(0,f_0)$ arbitrarily small in $W^{0,k-2,p}_{\boldsymbol\mu-2,\mathsf{Q}}((0,T)\times L)$ by making $\tau>0$ small. Thus for $\tau>0$ sufficiently small we have $w_{\tau}\in W$ and there exists $(u,f)\in V$ with $P(u,f)=w_{\tau}$. But then $P(u,f)=0$ on $(0,\tau)$ as we wanted to show.
\end{proof}

\subsection{Spatial regularity theory of the flow}
In this section we discuss the spatial regularity of solutions to $P(u,f)=0$. Detailed proofs of the results can be found in \cite[\S 9.6]{Behrndt1}.

We begin with the study of the spatial regularity of the function $u$.

\begin{Lem}\label{InteriorRegularity}
Let $(u,f)\in\mathcal{D}^{k,p}_{\boldsymbol\mu}$ be a solution of $P(u,f)=0$. Then $u(t,\cdot)\in C^{\infty}(L)$ for every $t\in(0,T)$.
\end{Lem}

\noindent The proof of Lemma \ref{InteriorRegularity} is more or less standard and not very exciting, so we skip it.

In the next lemma the decay rates of the higher derivatives of the function $u+a$ are studied, where $a\in C^{\infty}_{\boldsymbol\nu}(L)$ is given by Theorem \ref{LagrangianNeighbourhoodConicalSingularities}.

\begin{Lem}\label{XXXC}
Let $(u,f)\in\mathcal{D}^{k,p}_{\boldsymbol\mu}$ be a solution of $P(u,f)=0$. Assume that $u+a\in W^{1,2,p}_{\boldsymbol\gamma}((0,T)\times L)$ for some $\boldsymbol\gamma\in\mathbb{R}^n$ with $2<\boldsymbol\gamma<3$. Then $u+a\in W^{1,l,p}_{\boldsymbol\gamma}((0,T)\times L))$ for every $l\in\mathbb{N}$.
\end{Lem}

\noindent The proof of Lemma \ref{XXXC} is also not very exciting and merely uses some more or less standard techniques for linear parabolic equations on manifolds with conical singularities and Lemma \ref{InteriorRegularity}. Therefore we skip the proof of Lemma \ref{XXXC} as well.

In the next lemma we show that the rate of decay of the function $u$ becomes better for positive time $t>0$.

\begin{Lem}\label{ImprovedDecay}
Let $(u,f)\in\tilde{\mathcal{D}}^{k,p}_{\boldsymbol\mu}$ be a solution of the Cauchy problem \eq{CauchyProblem2}. Then $u+a\in W^{1,2,p}_{\boldsymbol\gamma}(I\times L)$ for every $I\subset\subset(0,T)$ and $\boldsymbol\gamma\in\mathbb{R}^n$ with $2<\boldsymbol\gamma<3$ and $(2,\gamma_i]\cap\mathcal{E}_{\Sigma_i}=\emptyset$ for $i=1,\ldots,n$.
\end{Lem}

\begin{proof}
Finally we have a lemma which is more interesting, so we will discuss the proof in some detail. The proof uses some more advanced techniques from functional analysis which can be found in Davies \cite{Davies}. Denote $v=u+a$. Since $\boldsymbol\mu<\boldsymbol\nu$, we have that $u+a\in W^{1,k,p}_{\boldsymbol\mu}((0,T)\times L)$. Denote $u_i=\phi_i^*(u)$ and define $v_i=u_i+a_i$ for $i=1,\ldots,n$. Then one can show that $v_i$ satisfies
\begin{gather}\label{WWW}
\begin{split}
&\frac{\partial v_i}{\partial t}(t,\sigma,r)=\Delta v_i(t,\sigma,r)+h_i(t,\sigma,r)\;\mbox{for }(t,\sigma,r)\in(0,T)\times\Sigma_i\times(0,R),\\
&v_i(0,\sigma,r)=a_i(\sigma,r)\;\quad\quad\quad\quad\quad\quad\quad\;\;\mbox{for }(\sigma,r)\in\Sigma_i\times(0,R)
\end{split}
\end{gather}
and $i=1,\ldots,n$, where the Laplace operator is taken with respect to the Riemannian cone metric $g_i=\iota_i^*(g')$ on $\Sigma_i\times(0,R)$ and $h_i:(0,T)\times\Sigma_i\times(0,R)\rightarrow\mathbb{R}$ is a function that is defined using the error terms $Q$ and $R_i$ from Lemma \ref{HELP} and the function $\Xi_{(u,f)}(\frac{\mathrm df}{\mathrm dt})$. The precise definition of $h_i$ is not important, the only important fact about $h_i$ is that if we choose some $h\in W^{0,k-2,p}((0,T)\times L)$ with $\phi_i^*(h)=h_i$ for $i=1,\ldots,n$, then $h\in W^{0,k-2,p}_{2\boldsymbol\mu-4}((0,T)\times L)$.

Now choose some Riemannian metric $\tilde{g}$ on $L$ with $\phi_i^*(g_i)=\tilde{g}$ for $i=1,\ldots,n$. Since $v_i$ satisfies \eq{WWW} for $i=1,\ldots,n$, we find that 
\begin{gather*}
\begin{split}
&\frac{\partial v}{\partial t}(t,x)=\Delta_{\tilde{g}}v(t,x)+h(t,x)+r(t,x)\quad\mbox{for }(t,x)\in(0,T)\times L,\\
&v(0,x)=a(x)\;\quad\quad\quad\quad\quad\quad\quad\quad\quad\quad\;\;\;\;\;\mbox{for }x\in L,
\end{split}
\end{gather*}
where the Laplace operator is taken with respect to the Riemannian metric $\tilde{g}$ and $r\in W^{0,k-2,p}((0,T)\times L)$ is supported on $(0,T)\times(L\backslash\bigcup_{i=1}^nS_i)$. Let $\tilde{H}$ be the Friedrichs heat kernel on $(L,\tilde{g})$. Then, by uniqueness of solutions to the heat equation, $v$ must be given by
\begin{equation*}
v(t,x)=\int_0^t\int_L\tilde{H}(t-s,x,y)(h+r)(s,y)\;\mathrm dV_{\tilde{g}}(y)\;\mathrm ds+\int_L\tilde{H}(t,x,y)a(y)\;\mathrm dV_{\tilde{g}}(y)
\end{equation*}
for $(t,x)\in(0,T)\times L$. Since $h\in W^{0,k-2,p}_{2\boldsymbol\mu-4}((0,T)\times L)$ one can use the same arguments as in the proof of Theorem \ref{SobolevRegularityHeat} to show that the first term lies in $W^{1,k,p}_{2\boldsymbol\mu-2,\mathsf{P}^{\Delta}_{2\boldsymbol\mu-2}}((0,T)\times L)$. Moreover by the standard mapping properties of the Friedrich heat kernel and the Sobolev Embedding Theorem it follows that the second term lies in $C^{\infty}_{\mathsf{P}^{\Delta}_{\boldsymbol\delta}}(L)$ for every $t\in(0,T)$ and $\boldsymbol\delta\in\mathbb{R}$ and is smooth in $t\in(0,T)$. Hence it follows that
\[
v\in W^{1,k,p}_{\boldsymbol\mu}(I\times L)\cap W^{1,k,p}_{2\boldsymbol\mu-2,\mathsf{P}^{\Delta}_{2\boldsymbol\mu-2}}(I\times L)
\]
for every $I\subset\subset(0,T)$. In particular, if $(2,2\mu_i-2]\cap\mathcal{E}_{\Sigma_i}=\emptyset$ for $i=1,\ldots,n$, then it follows that $v\in W^{1,k,p}_{2\boldsymbol\mu-2}(I\times L)$ for every $I\subset\subset(0,T)$. In particular, since $\boldsymbol\mu>2$, we have that $2\boldsymbol\mu-2>\boldsymbol\mu$ and thus we have improved the decay rate of the function $v$. Iterating this procedure we then find that $v\in W^{1,k,p}_{\boldsymbol\gamma}(I\times L)$ for every $I\subset\subset(0,T)$ and every $\boldsymbol\gamma\in\mathbb{R}^n$ with $2<\boldsymbol\gamma<3$ and $(2,\gamma_i]\cap\mathcal{E}_{\Sigma_i}=\emptyset$ for $i=1,\ldots,n$.
\end{proof}

This completes our study of the spatial regularity of solutions to $P(u,f)=0$. We summarize the previous three lemmas and some obvious conclusions in the following proposition.

\begin{Prop}\label{FinalRegularity}
Let $(u,f)\in\mathcal{D}^{k,p}_{\boldsymbol\mu}$ be a solution of the Cauchy problem \eq{CauchyProblem2}. Then $f$ defines $W^{1,p}$-one-parameter families $\{x_i(t)\}_{t\in(0,T)}$ of points in $M$ for $i=1,\ldots,n$ and of isomorphisms $\{A_i(t)\}_{t\in(0,T)}$ for $i=1,\ldots,n$ with $A_i(t)\in\mathcal{A}_{x_i(t)}$ for $i=1,\ldots,n$. Finally define a one-parameter family $\{F(t,\cdot)\}_{t\in(0,T)}$ of Lagrangian submanifolds as in Proposition \ref{HHUU}. Then $\{F(t,\cdot)\}_{t\in(0,T)}$ is a $W^{1,p}$-one-parameter family of smooth Lagrangian submanifolds with isolated conical singularities modelled on $C_1,\ldots,C_n$. For $t\in(0,T)$ the Lagrangian submanifold $F(t,\cdot):L\rightarrow M$ has conical singularities $x_1(t),\ldots,x_n(t)$ and isomorphisms $A_i(t)\in\mathcal{A}_{x_i(t)}$ for $i=1,\ldots,n$ as in Definition \ref{DefLagrangianConicalSingularities}. Moreover for every $t\in(0,T)$, $F(t,\cdot):L\rightarrow M$ satisfies \eq{ConeCondition} for every $\boldsymbol\gamma\in\mathbb{R}$ with $2<\boldsymbol\gamma<3$ and $(2,\gamma_i]\cap\mathcal{E}_{\Sigma_i}=\emptyset$ for $i=1,\ldots,n$.
\end{Prop}

\subsection{What about the time regularity of the flow?}
So far we have only discussed the spatial regularity of $u$ and not how the regularity of $u$ and $f$ in time direction can be improved. At the present state the author has no idea how the time regularity of $u$ and $f$ can be improved. Why is there a problem with the time regularity of solutions to $P(u,f)=0$, when $u\in W^{1,k,p}_{\boldsymbol\mu}((0,T)\times L)$ and $f\in W^{1,p}((0,T);\mathcal{F})$? The author is aware of two methods how time regularity for solutions of parabolic equations can be improved, but neither method seems to work in our case.

The first method would be to differentiate the equation $P(u,f)=0$ with respect to $t$ and then to use standard regularity theory for linear equations to show that $u$ and $f$ have one more time derivative than apriori known. Let us see what happens when we differentiate the equation $P(u,f)=0$ with respect to $t$. By differentiating with respect to $t$ we find that
\[
\frac{\partial}{\partial t}\left(\partial_tu-\Xi_{(u,f)}(\mbox{$\frac{\mathrm df}{\mathrm dt}$})\right)=\Delta\left(\partial_tu-\Xi_{(u,f)}(\mbox{$\frac{\mathrm df}{\mathrm dt}$})\right)+R(f,u,\mathrm du,\nabla\mathrm du,\partial_tu,\mbox{$\frac{\mathrm df}{\mathrm dt}$}),
\]
where the force term $R$ is some smooth function on its domain. The only important fact about the $R$-term is that it only depend on the first time derivatives of $u$ and $f$. In particular the $R$-term is $L^p$ in time. Using standard regularity theory for linear parabolic equations it follows that the function $\partial_tu-\Xi_{(u,f)}(\mbox{$\frac{\mathrm df}{\mathrm dt}$})$ is $W^{1,p}$ in time. So far so good, but what we really want is that $\partial_tu$ and $\Xi_{(u,f)}(\mbox{$\frac{\mathrm df}{\mathrm dt}$})$ are $W^{1,p}$ in time. What is the problem? Notice that when we differentiate $u$ with respect to $t$, then we lose two rates of decay, and therefore, loosely speaking, we have that $\partial_tu=O(\rho^{\boldsymbol\mu-2})$. In particular $\partial_tu$ and $\Xi_{(u,f)}(\mbox{$\frac{\mathrm df}{\mathrm dt}$})$ do not lie in complementary spaces anymore and we are not able to conclude that both $\partial_tu$ and $\Xi_{(u,f)}(\mbox{$\frac{\mathrm df}{\mathrm dt}$})$ are $W^{1,p}$ in time from knowing that $\partial_tu-\Xi_{(u,f)}(\mbox{$\frac{\mathrm df}{\mathrm dt}$})$ is $W^{1,p}$ in time.

The second method to improve time regularity of solutions to parabolic equations is to write the nonlinear parabolic equation as the heat equation plus a nonlinear perturbation term and then again to use regularity theory for linear parabolic equations to improve the regularity. So, more or less, we would like to write $P(u,f)=0$ in the following form
\begin{equation}\label{FFFGGGSSS}
\frac{\partial}{\partial t}(u+v)=\Delta(u+v)+R(\mathrm du,\nabla\mathrm du,\mathrm dv,\nabla\mathrm dv),
\end{equation}
where $v$ is the discrete asymptotics part and $R$ is some smooth function on its domain. Notice in particular that the $R$-term is $W^{1,p}$ in time, since it does not involve any time derivatives of $u$ or the discrete asymptotics part $v$. If we were able to write $P(u,f)=0$ in this form, then we one could use the heat kernel to write down explicit formul{\ae} for $u$ and $v$. In fact, if $H$ is the heat kernel on $L$, then $H$ admits a decomposition $H=H_{\boldsymbol\gamma}+H_{\mathsf{P}_{\boldsymbol\gamma}^{\Delta}}$, and then $u$ is given by 
\[
u(t,x)=\int_0^t\int_LH_{\boldsymbol\gamma}(t-s,x,y)R(s,y)\;\mathrm dV_g(y)\;\mathrm ds,
\]
and the discrete asymptotics part $v$ is given by
\[
v(t,x)=\int_0^t\int_LH_{\mathsf{P}_{\boldsymbol\gamma}^{\Delta}}(t-s,x,y)R(s,y)\;\mathrm dV_g(y)\;\mathrm ds,
\]
where we write $R(s,y)$ for $R(\mathrm du(s,y),\nabla\mathrm du(s,y),\mathrm dv(s,y),\nabla\mathrm dv(s,y))$ and $(s,y)\in(0,T)\times L$. 
Since the $H_{\boldsymbol\gamma}$ and $H_{\mathsf{P}_{\boldsymbol\gamma}^{\Delta}}$ are smooth in time and $R$ is $W^{1,p}$ in time, one can then show that both, $u$ and $v$, are $W^{2,p}$ in time. However, notice carefully that the equation $P(u,f)=0$ does not really have the same form as equation \eq{FFFGGGSSS}. In fact, instead of having a time derivative of the discrete asymptotics part, we have the term $\Xi_{(u,f)}(\frac{\mathrm df}{\mathrm dt})$, where $\Xi_{(u,f)}:f^*(T\mathcal{F})\rightarrow V_{\mathsf{P}_{(u,f)}}(L)$ is a time dependent map of vector bundles. Therefore, if we try to pull out the time derivative in $\Xi_{(u,f)}(\frac{\mathrm df}{\mathrm dt})$ and try to write this term in the form ``time derivative of the discrete asymptotics part," then we get new terms that involve the time derivatives of $u$ and $f$ because we have to differentiate the time dependent bundle map $\Xi_{(u,f)}$.

\subsection{The main result}
Combining Propositions \ref{HHUU}, \ref{ExistenceSolutions}, and \ref{FinalRegularity} we conclude our main theorem about the short time existence of the generalized Lagrangian mean curvature flow, when the initial Lagrangian submanifold has isolated conical singularities modelled on stable special Lagrangian cones.

\begin{Thm}\label{MainTheorem}
Let $(M,J,\omega,\Omega)$ be an $m$-dimensional almost Calabi--Yau manifold, $m\geq 3$, $C_1,\ldots,C_n$ stable special Lagrangian cones in $\mathbb{C}^m$, and $F_0:L\rightarrow M$ a Lagrangian submanifold with isolated conical singularities at $x_1,\ldots,x_n$, modelled on the stable special Lagrangian cones $C_1,\ldots,C_n$ as in Definition \ref{DefLagrangianConicalSingularities}. Then for sufficiently large $p\in(1,\infty)$ there exists $T>0$, $W^{1,p}$-one-parameter families of points $\{x_i(t)\}_{t\in(0,T)}$ in $M$ for $i=1,\ldots,n$, continuous up to $t=0$, with $x_i(0)=x_i$ for $i=1,\ldots,n$, and $W^{1,p}$-one-parameter families $\{A_i(t)\}_{t\in(0,T)}$ of isomorphisms $A_i(t)\in\mathcal{A}_{x_i(t)}$ for $i=1,\ldots,n$, continuous up to $t=0$, with $A_i(0)=A_i$ for $i=1,\ldots,n$, such that the following holds.

There exists a $W^{1,p}$-one-parameter family $\{F(t,\cdot)\}_{t\in(0,T)}$ of smooth Lagrangian submanifolds $F(t,\cdot):L\rightarrow M$, continuous up to $t=0$, with isolated conical singularities at $x_1(t),\ldots,x_n(t)$ modelled on the special Lagrangian cones $C_1,\ldots,C_n$ and with isomorphisms $A_1(t),\ldots,A_n(t)$, $A_i(t):\mathbb{C}^m\rightarrow T_{x_i(t)}M$ for $i=1,\ldots,n$ as in Definition \ref{DefLagrangianConicalSingularities}, which evolves by generalized Lagrangian mean curvature flow with initial condition $F_0:L\rightarrow M$. Moreover, for every $t\in(0,T)$ the Lagrangian submanifold $F(t,\cdot):L\rightarrow M$ satisfies \eq{ConeCondition} for every $\boldsymbol\gamma\in\mathbb{R}^n$ with $\gamma_i\in(2,3)$ and $(2,\gamma_i]\cap\mathcal{E}_{\Sigma_i}=\emptyset$ for $i=1,\ldots,n$.
\end{Thm}

{\footnotesize\noindent LINCOLN COLLEGE, TURL STREET, OX1 3DR, OXFORD, UNITED KINGDOM}\\
{\small\verb|tapio.behrndt@gmail.com|}
\end{document}

%% file: LMCF.pstex_t
\begin{picture}(0,0)%
\includegraphics{lmcf.pstex}%
\end{picture}%
\setlength{\unitlength}{3947sp}%
\begingroup\makeatletter\ifx\SetFigFont\undefined%
\gdef\SetFigFont#1#2#3#4#5{%
  \reset@font\fontsize{#1}{#2pt}%
  \fontfamily{#3}\fontseries{#4}\fontshape{#5}%
  \selectfont}%
\fi\endgroup%
\begin{picture}(8293,3306)(654,-2937)
\put(2126,-2861){\makebox(0,0)[lb]{\smash{{\SetFigFont{12}{14.4}{\rmdefault}{\mddefault}{\updefault}{\color[rgb]{0,0,0}$x(0)$}%
}}}}
\put(2601,-2336){\makebox(0,0)[lb]{\smash{{\SetFigFont{12}{14.4}{\rmdefault}{\mddefault}{\updefault}{\color[rgb]{0,0,0}$C(0)$}%
}}}}
\put(7401,-2711){\makebox(0,0)[lb]{\smash{{\SetFigFont{12}{14.4}{\rmdefault}{\mddefault}{\updefault}{\color[rgb]{0,0,0}$x(0)$}%
}}}}
\put(6439,-2011){\makebox(0,0)[lb]{\smash{{\SetFigFont{12}{14.4}{\rmdefault}{\mddefault}{\updefault}{\color[rgb]{0,0,0}$C(0)$}%
}}}}
\put(3189,189){\makebox(0,0)[lb]{\smash{{\SetFigFont{12}{14.4}{\rmdefault}{\mddefault}{\updefault}{\color[rgb]{0,0,0}$F_0:L\longrightarrow M$}%
}}}}
\put(5114,189){\makebox(0,0)[lb]{\smash{{\SetFigFont{12}{14.4}{\rmdefault}{\mddefault}{\updefault}{\color[rgb]{0,0,0}$F(t,\cdot):L\longrightarrow M$}%
}}}}
\put(8239,-2574){\makebox(0,0)[lb]{\smash{{\SetFigFont{12}{14.4}{\rmdefault}{\mddefault}{\updefault}{\color[rgb]{0,0,0}$x(t)$}%
}}}}
\put(8439,-1864){\makebox(0,0)[lb]{\smash{{\SetFigFont{12}{14.4}{\rmdefault}{\mddefault}{\updefault}{\color[rgb]{0,0,0}$C(t)$}%
}}}}
\end{picture}%